\documentclass[english]{article}
\usepackage[T1]{fontenc}
\usepackage[latin9]{inputenc}
\usepackage{verbatim}
\usepackage{amsmath}
\usepackage{amssymb}

\makeatletter
\usepackage{amsthm,amsmath,lineno}
\usepackage{array}\usepackage{enumerate}\usepackage{nicefrac}\usepackage{tikz}\usepackage{listings}\newcounter{count}
\newtheorem{thm}[count]{Theorem}
\newtheorem{lem}[count]{Lemma}
\newtheorem{prop}[count]{Proposition}

\newtheorem{conj}[count]{Conjecture}
\newtheorem{cor}[count]{Corollary}
\theoremstyle{definition}
\newtheorem{defn}[count]{Definition}
\newtheorem{rem}[count]{Remark}

\newtheorem{quest}[count]{Question}

\makeatother

\usepackage{babel}
\begin{document}
\global\long\def\ZZ{\mathbb{Z}}
\global\long\def\S{\mathcal{S}}
\global\long\def\A{\mathcal{A}}
\global\long\def\G{\mathcal{G}}
\global\long\def\As#1#2#3{\A_{#1,#2,#3}}
\global\long\def\B{\mathcal{B}}
\global\long\def\N{\mathcal{N}}
\global\long\def\C{\mathcal{C}}

%\linenumbers

\title{EKR sets for large $n$ and $r$} 
%\author{Ben Bond}
\author{Ben Bond}
\maketitle

\begin{abstract}
Let $\A\subset\binom{[n]}{r}$ be a compressed, intersecting family
and let $X\subset[n]$. Let $\A(X)=\{A\in\A:A\cap X\ne\emptyset\}$
and $\S_{n,r}=\binom{[n]}{r}(\{1\})$. Motivated by the Erd\H{o}s-Ko-Rado
theorem, Borg asked for which $X\subset[2,n]$ do we have $|\A(X)|\le|\S_{n,r}(X)|$
for all compressed, intersecting families $\A$? We call $X$ that
satisfy this property \emph{EKR}. Borg classified EKR sets $X$ such
that $|X|\ge r$. Barber classified $X$, with $|X|\le r$, such that
$X$ is EKR for sufficiently large $n$, and asked how large $n$
must be. We prove $n$ is sufficiently large
when $n$ grows quadratically in $r$. In the case where $\A$ has
a maximal element, we are able to sharpen this bound to $n>\varphi^{2}r$
implies $|\A(X)|\le|\S_{n,r}(X)|$. We conclude by giving a generating
function that speeds up computation of $|\A(X)|$ in comparison with
the na\"{i}ve methods.
\end{abstract}

\section{Introduction \label{s.intro}}
The main objects of study in this paper are compressed, intersecting
families. We begin by defining these terms. Let $\binom{[n]}{r}$
denote the set of $r$ element subsets of $[n]=\{1,\ldots,n\}$. We
label elements of $\binom{[n]}{r}$ in increasing order i.e. for $B=\{b_{1},\ldots,b_{r}\}\in\binom{[n]}{r}$
we have $b_{i}<b_{i+1}$. A \emph{family} $\A$ is a subset $\A\subset\binom{[n]}{r}$.
We say $\A$ is \emph{intersecting} if $B,C\in\A$ implies $B\cap C\ne\emptyset$.
Notice that $\A$ is trivially intersecting if $n<2r$ by the pigeonhole
principle. This makes it possible to define a partial order known
as the \emph{compression} \emph{order} on $\binom{[n]}{r}$, as follows.
For $A=\{a_{1},\ldots,a_{r}\}$, and $B=\{b_{1},\ldots,b_{r}\}$,
we define $A\le B$ if $a_{i}\le b_{i}$ for all $1\le i\le r$. We
say a family $\A$ is \emph{compressed} if $A\in\A$ implies $B\in\A$
for $B\le A$. We extend this partial order to $2^{[n]}$: for $C=\{c_{1},\ldots,c_{k}\}$,
we say $A\prec C$ if $r\ge k$ and $a_{i}\le c_{i}$ for $1\le i\le k$.
For example, $\{1,2,3\}\prec\{1,2\}$ but $\{1,2\}\not\prec\{1,2,3\}$.

Let $\A$ be a family and $X\subset[n]$. One of the main objects
of study in this paper is $\A(X)$, which we define by $\A(X)=\{A\in\A:A\cap X\ne\emptyset\}$.
An important example of such a family is $\S_{n,r}$, defined by 
\[
\S_{n,r}=\binom{[n]}{r}(\{1\})=\{A\in\binom{[n]}{r}:1\in A\}.
\]
We will denote $\S_{n,r}$ by $\S$ if $n$ and $r$ are clear. It
is easy to check that $\S$ is compressed and intersecting.

The following theorem is one of the fundamental results about intersecting
families.

\begin{thm}(Erd\H{o}s-Ko-Rado)\label{ekr}\cite{ekr} (see also \cite{ff})
Let $n\ge2r$ and let $\A\subset\binom{[n]}{r}$ be an intersecting
family. Then $|\A|\le|\S|$. \end{thm}

\noindent In \cite{borg}, Borg considered a variant of the Erd\H{o}s-Ko-Rado
theorem. Borg asked which sets $X\subset[2,n]$ have the property
that $|\A(X)|\le|\S_{n,r}(X)|$ for all compressed, intersecting families
$\A$. We call $X$ with this property \emph{EKR.} We assume $X\subset[2,n]$,
because if $1\in X$, then $\S(X)=\S$ and $X$ is trivially EKR by
the Erd\H{o}s-Ko-Rado theorem. There are many $X$ which are not
EKR. For example, consider the Hilton-Milner family $\N=\S([2,r+1])\cup\{[2,r+1]\}$
\cite{hm} . Then for $X=[2,r+1]$, we have $|\N(X)|=|\S(X)|+1$. 

The motivation for considering compressed families is twofold. Firstly,
the question is uninteresting without the requirement that $\A$ be
compressed, since for any $x\in X$, we have $\binom{[n]}{r}(\{x\})$
maximizes $|\A(X)|$ for intersecting families $\A$. Secondly, arbitrary
sets lack structure, and by imposing more conditions, we may gain
more information. In fact, compressed families and the shifting technique
(see \cite{frankl} for a survey) are powerful techniques in extremal
set theory and can be used to give a simple proof of the Erd\H{o}s-Ko-Rado
theorem.

In \cite{borg}, Borg classified $X$ that are EKR for $|X|\ge r$
and gave a partial solution in the case $|X|<r$. Barber continued
with Borg's work in \cite{barber} by considering $|X|\le r$. To
describe his results, we introduce the notion of \emph{eventually
EKR} sets, which are sets $X\subset[2,n]$ such that for fixed $r$,
we have that $X$ is EKR for sufficiently large $n$.

\begin{thm}(Barber)\label{t.barber} Let $r\ge3$, $n\ge2r$ and
$X\subset[2,n]$ with $|X|\le r$. If $X\not\subseteq[2,r+1]$, then
$X$ is eventually EKR if and only if one of the following holds

\begin{enumerate}
\item $|X|=1$
\item $|X|=2$ and $2,3\notin X$
\item $|X|=3$ and $\{2,3\}\not \subset X$
\item $|X|\ge 4.$
\end{enumerate}
\end{thm}\noindent Barber asked which $n$ are sufficiently large to imply
$X$ is EKR. This paper provides bounds on $n$. 

Based on numerical results for small $n$ and $r$ stated in \cite{barber},
Barber speculated that $n\ge2r+2$ was sufficient to imply $X$ is
EKR. However, as will be seen in Section \ref{s.bigbad}, this bound
does not hold in general. We replace the suggested bound of $n\ge2r+2$
with the following conjecture, which is supported by computer evidence
for $r\le5$.

\begin{conj}\label{c.phicon} Let $r\ge2$ and $X\subset[2,n]$ be
eventually EKR. Then $n>\varphi^{2}r$ implies $X$ is EKR, where
$\varphi=\frac{1+\sqrt{5}}{2}$. \end{conj}

Note that using the discussion after Theorem 4 in \cite{barber},
it is easy to show the conjecture holds in the case $r=2$, so we
address the case $r\ge3$.

In order to describe our results towards Conjecture \ref{c.phicon},
we need the notion of generating sets. These were introduced by Ahlswede
and Khachatrian in \cite{ak}, and Barber considered a variant definition,
which is more useful in the present context. Let $\G\subset2^{[n]}$.
We define 
\[
F(r,n,\G)=\{A\in\binom{[n]}{r}:A\prec G\mbox{ for some }G\in\G\}.
\]

A \emph{maximal} compressed, intersecting family $\A$ is one that
is not properly contained in a compressed, intersecting family. Barber
proved in \cite{barber} that every maximal compressed, intersecting
$\A$ can be expressed as $F(r,n,\G)$ for some $\G$. Thus to show
that a given $X$ is EKR, it suffices to show $|\A(X)|\le|\S(X)|$
for intersecting families $\A$ of the form $\A=F(r,n,\G)$. Observe
that such families are naturally compressed. We may now state the
main theorems of this paper. 

\begin{thm}\label{t.multgen} Let $r\ge3$ and let $X$ be eventually
EKR with $|X|\le r$. For each $\varepsilon>0$, there exists an $r_{0}$
such that for $r>r_{0}$, the condition $n>(2+\varepsilon)r^{2}$
implies $X$ is EKR. Furthermore, 

\begin{enumerate}
\item If $|X|=1$ and $r\ge 12$, then $n>2r^2$ implies $X$ is EKR.
\item If $|X|=2$ and $r\ge 14$, then $n>2r^2$ implies $X$ is EKR.
\item If $|X|=3$ and $r\ge 14$, then $n>3r^2$ implies $X$ is EKR.
\item If $|X|=t\ge 4$ and $r\ge \max\{11,t\}$, then $n>tr^2$ implies $X$ is EKR.
\end{enumerate}
\end{thm}In the case of a single generator, we have a sharper bound, which
provides evidence for Conjecture \ref{c.phicon}.

\begin{thm}\label{t.onegen} Let $r\ge3$, let $\A=F(r,n,\G)$ be
an intersecting family with $|\G|=1$, and let $X$ be eventually
EKR. Then $n>\varphi^{2}r$ implies $|\A(X)|\le|\S(X)|$. \end{thm}

The outline of the paper is as follows. In Section \ref{s.bigbad}
we give an example of a family $\A$ and set $X$ showing that the
coefficient $\varphi^{2}$ in Conjecture \ref{c.phicon} cannot be
made any smaller. Section \ref{s.intcond} gives a necessary and sufficient
condition for a compressed family to be intersecting. This will be
useful in the sections that follow. Section \ref{s.onegenprelims}
establishes preliminaries necessary in the proof of Theorem \ref{t.onegen},
which appears in Section \ref{s.onegen}. We then use the results
from Sections \ref{s.onegenprelims} and \ref{s.onegen} to prove
Theorem \ref{t.multgen} in Section \ref{s.multgen}. Finally, in
Section \ref{s.genfun} we give a generating function that greatly
speeds up numerical computations for $|F(r,n,\G)(X)|$ in comparison
with the na\"{i}ve methods.

\section{A family for which $|\A(X)|>|\S(X)|$\label{s.bigbad}}

In this section, we exhibit a family that shows tightness of the coefficient
$\varphi^{2}$ in the bound $n>\varphi^{2}r$ of Conjecture \ref{c.phicon}.

\begin{prop}Let $X=\{2,4,r+2\}$, $\A=F(r,n,\{\{2,3\}\})$. For $r\ge4$,
we have $n<\frac{3r+1+\sqrt{5r^{2}-22r+25}}{2}$ implies $|\A(X)|>|\S(X)|$.

\end{prop}

\begin{proof}

We begin by computing $|\A(X)|$ and $|\S(X)|$. To compute $|\A(X)|$,
notice that the first two elements of any $A\in\A$ must be $\{1,2\}$,
$\{1,3\}$, or $\{2,3\}$. There are $\binom{n-2}{r-2}$ elements
of $\A(X)$ with first two elements $\{1,2\}$ and $\binom{n-3}{r-2}$
elements of $\A(X)$ with first two elements $\{2,3\}$, since in
both cases such sets contain $2\in X$. We then use the principle
of inclusion-exclusion to count $2\binom{n-4}{r-3}-\binom{n-5}{r-4}$
elements that have first two elements $\{1,3\}$ and contain $4$
or $r+2$. Thus 
\[
|A(X)|=\binom{n-2}{r-2}+\binom{n-3}{r-2}+2\binom{n-4}{r-3}-\binom{n-5}{r-4}
\]

To count $|\S(X)|$, notice that this is just counting $(r-1)$ element
subsets of an $(n-1)$ element set (since $1$ is already accounted
for) that contain at least one of three distinguished elements. Thus
we may use the principle of inclusion-exclusion to get
\[
|\S(X)|=3\binom{n-2}{r-2}-3\binom{n-3}{r-3}+\binom{n-4}{r-4}
\]
To find when $|\A(X)|>|\S(X)|$, we compare the two expressions and
simplify to obtain the inequality
\begin{eqnarray*}
0>|\S(X)|-|\A(X)| & = & 3\binom{n-2}{r-2}-3\binom{n-3}{r-3}+\binom{n-4}{r-4}\\
 &  & -\binom{n-2}{r-2}-\binom{n-3}{r-2}-2\binom{n-4}{r-3}+\binom{n-5}{r-4}\\
 & = & \frac{(n-5)!}{(r-2)!(n-r)!}(2(n-2)(n-3)(n-4)-3(n-3)(n-4)(r-2)+\\
 &  & (n-4)(r-2)(r-3)-(n-3)(n-4)(n-r)-\\
 &  & 2(n-4)(r-2)(n-r)+(r-2)(r-3)(n-r)
\end{eqnarray*}
Multiplying by $(r-2)!(n-r)!/(n-5)!$ and expanding, we get
\[
0>n^{3}+(-4r-1)n^{2}+(4r^{2}+8r-6)n-r^{3}-7r^{2}+6r
\]
Notice that the right hand side is divisible by $n-r$. Factoring
gives
\[
0>(n-r)(n^{2}+(-3r-1)n+r^{2}+7r-6)
\]
or
$$n<\frac{3r+1+\sqrt{5r^{2}-22r+25}}{2},$$
as desired.\end{proof}

Notice that $\frac{3r+1+\sqrt{5r^{2}-22r+25}}{2}=\varphi^{2}r+o(r)$,
so the bound given in Conjecture \ref{c.phicon} is tight, up to lower
order terms.

\section{Conditions for a family to be intersecting \label{s.intcond}}

In this section, we determine necessary and sufficient conditions
for a family to be intersecting that will be useful later in the paper.
The purpose of this section is to prove the following proposition.

\begin{prop}\label{p.intcond} Let $\A$ be a compressed family. $\A$
is intersecting if and only if for any $A=\{a_{1},\ldots,a_{r}\}$,
$B=\{b_{1},\ldots b_{r}\}\in\A$ there exists a pair $i,j$, with
$1\le i,j\le r$ such that $i+j>\max\{a_{i},b_{j}\}$.\end{prop}

Note that this is especially useful in the case where $\A=F(r,n,\G)$.
In this case assume that for some $A,B\in\A$, we have an $i,j$ such
that $i+j>\max\{a_{i},b_{j}\}$, then for any $C=\{c_{1},\ldots,c_{r}\}\le A$
and $D=\{d_{1},\ldots,d_{r}\}\le B$, we have $i+j>\max\{c_{i},d_{j}\}$,
thus it is sufficient to find such a pair $i,j$ for each pair of
generators.

Similar results may also be found in Section 8 of \cite{frankl}. In fact, the ``only if'' of Proposition \ref{p.intcond} follows from Proposition 8.1 of \cite{frankl}. However, Lemma \ref{l.intcond1} will be important in the proof of Theorem \ref{t.onegen}, so we present it separately, and only use Proposition 8.1 in Lemma \ref{l.intcond2}.

\begin{lem}\label{l.intcond1} Let $\A=F(r,n,\{A\})$, with $A=\{a_{1},\ldots,a_{r}\}$.
Then $\A$ is intersecting if and only if $A\prec[s,2s-1]$ for some
$s$, with $1\le s\le r$. \end{lem}

\begin{proof}

If $A\prec[s,2s-1]$, then $\A$ is intersecting by an easy application
of the pigeonhole principle. 

Suppose that $A\not\prec[s,2s-1]$ for all $s$. We claim $a_{i}\ge2i$
for $1\le i\le r$. To see this, note that if $A\not\prec[s,2s-1]$,
then there exists some $i$, with $1\le i\le s$ such that $a_{i}>s+i-1$.
Since $a_{i+1}>a_{i}$, this implies $a_{s}>2s-1$. Since this holds
for each $s\in[1,r]$, we have $a_{i}\ge2i$ for $1\le i\le r$. From
this, we see that $\{2,4,6,\ldots2r\}\le A$ and $\{1,3,5,\ldots,2r-1\}\le A$
are nonintersecting. Hence for $\A$ to be intersecting, we must have
$A\prec[s,2s-1]$ for some $s$. \end{proof}

The remaining lemmas are most easily stated with the following definition.

\begin{defn} Let $n\ge2r$, $A=\{a_{1},\ldots,a_{r}\},\, B=\{b_{1},\ldots,b_{r}\}$
be ordered sets in $\binom{[n]}{r}$. We say $A$ and $B$ are \emph{cross
intersecting} if $C\cap D\ne\emptyset$ for any $C\le A$, and $D\le B$.\end{defn}

\begin{lem}\label{l.intcond2} Sets $A,B\in\binom{[n]}{r}$ are cross
intersecting if and only if there exists a pair $i,j$, with $1\le i,j\le r$
such that $i+j>\max\{a_{i},b_{j}\}$. \end{lem}

\begin{proof}

Let $C=\{c_{1},\ldots,c_{r}\}\le A$ and $D=\{d_{1},\ldots,d_{r}\}\le B$,
and assume there exists a pair $i,j$ with $i+j>a_{i},b_{j}$. Without
loss of generality, assume $a_{i}\le b_{j}$. Then $c_{1},\ldots,c_{i},d_{1},\ldots,d_{j}\le b_{j}$,
so we have $i+j>b_{j}$ elements that are less than or equal to $b_{j}$,
so two must be the same by the pigeonhole principle. $ $Thus in this
case, $A$ and $B$ are cross intersecting.

The opposite direction is a direct consequence of Proposition 8.1 of \cite{frankl} with $r=2$, $\mathcal{F}_1=F(r,n,\{A\})$, and $\mathcal{F}_2=F(r,n,\{B\})$. We simply take $i=|A\cap [1,\ell]|$, and $j=|B\cap[1,\ell]|$.
 \end{proof}

\begin{lem}\label{l.intcond3} Let $A=\{a_{1},\ldots,a_{r}\}$. There
exists an $s$ such that $A\prec[s,2s-1]$ if and only if the pair
$A,A$ is cross intersecting. \end{lem}

\begin{proof}

If $A\prec[s,2s-1]$, then $a_{s}<2s$, so the pair $i,j=s$ shows
$A,A$ is cross intersecting. Conversely, assume there exists $i$
and $j$ such that $i+j>a_{i},a_{j}$. Without loss of generality,
$i\ge j$, so 
\[
2i\ge i+j>a_{i}.
\]
This is sufficient to show $A\prec[i,2i-1]$, since $a_{k}\le a_{k+1}-1$.\end{proof}

\begin{proof}[Proof of Proposition \ref{p.intcond}]For any $A,B\in\A$,
we have $F(r,n,\{A,B\})\subset\A$, so the proof is an immediate consequence
of the Lemmas \ref{l.intcond1}, \ref{l.intcond2}, and \ref{l.intcond3}.\end{proof}

\section{Preliminaries for the proof of Theorem \ref{t.onegen} \label{s.onegenprelims}}

In this section, we establish preliminaries that will be necessary
for the proof of Theorem \ref{t.onegen} in Section \ref{s.onegen}.
Many of these same techniques will be used in the proof of Theorem
\ref{t.multgen} in Section \ref{s.multgen} as well. We begin with
an outline of the proof of Theorem \ref{t.onegen}.

Let $\A_{n,r,s}=F(r,n,\{[s,2s-1]\})$. Notice that by Lemma \ref{l.intcond1},
it is sufficient to consider families of this form. We begin by computing
an explicit formula for $f_{X}(n,r,s)=|\A_{n,r,s}(X)|$ and then we
substitute $n=ar$. In each of the cases, we determine for which $a$,
independent of $s$, we have 
\[
\lim_{r\rightarrow\infty}\frac{f_{X}(ar,r,s)}{|\S_{ar,r}(X)|}<1
\]

In each case, we will show this limit is indeed less than 1 for $a>\varphi^{2}$.
After we have done this, we will re-analyze our formulas to find how
large $r$ must be so that $\frac{f(ar,r,s)}{|S_{ar,r}(X)|}<1$. 

The function $f_{X}$ will break into two parts, with a subsection
devoted to each part. The first part (Subsection \ref{ss.sink}) is
negligibly small in comparison with $|\S(X)|$ for large $r$, which
we will show using Stirling's approximation. The other part (Subsection
\ref{ss.decins}) is decreasing in $s$, so we only need to evaluate
the above limit for small values of $s$. 

In \cite{borg}, Borg proved that for a compressed family $\A$, and
$X,X'\subset[n]$ with $X'\ge X$ we have $|\A(X)|\ge|\A(X')|$. Notice
that $\{r+2\}$, $\{4,r+2\}$, $\{2,4,r+2\}$ and $\{2,\ldots,|X|,r+2\}$,
are the minimal elements in each of the cases of Theorem \ref{t.barber},
thus it suffices to assume $X$ is one of these sets. This is important,
because it greatly reduces the number of cases we need to consider. 

To begin, we need to count $|\A_{n,r,s}(X)|$ for these $X$.

\subsection{Counting $|\A_{n,r,s}(X)|$ \label{ss.countA}}

\begin{lem}\label{l.countform}

Let $X=\{r+2\},\,\{4,r+2\},\,\{2,4,r+2\}$ or $\{2,3,\ldots,t,r+2\}$,
where $t=|X|$, and let $r\ge4$.

\noindent If $t=2$, $s=2$, then
\[
|\A_{n,r,2}(X)|=2\binom{n-3}{r-3}-\binom{n-4}{r-4}+2\left(2\binom{n-4}{r-3}-\binom{n-5}{r-4}\right).
\]
If $t=2$, $s=3$, then
\[
|\A_{n,r,3}(X)|=2\binom{n-4}{r-4}-\binom{n-5}{r-5}+3\binom{n-4}{r-3}+3\binom{n-6}{r-4}+3\binom{n-5}{r-3}.
\]
If $t=3$, $s=2$, then 
\[
|\A_{n,r,2}(X)|=\binom{n-2}{r-2}+\binom{n-3}{r-2}+2\binom{n-4}{r-3}-\binom{n-5}{r-4}.
\]
If $t=3$, $s=3$, then 
\[
|\A_{n,r,3}(X)|=\binom{n-3}{r-3}+3\binom{n-4}{r-3}+5\binom{n-5}{r-3}+\binom{n-6}{r-4}.
\]
Otherwise, we have
\begin{eqnarray}
|\A_{n,r,s}(X)| & = & \sum_{i=s}^{2s-1}\left(\binom{i-1}{s-1}-\binom{i-t}{s-1}\right)\binom{n-i}{r-s}+\sum_{i=s}^{\min\{r+1,2s-1\}}\binom{i-t}{s-1}\binom{n-i-1}{r-s-1}\label{eq:countform}\\
 &  & +\binom{r+2-t}{s-1}\binom{n-r-2}{r-s}+\sum_{i=r+3}^{2s-1}\binom{i-t-1}{s-2}\binom{n-i}{r-s}\nonumber 
\end{eqnarray}
where $\binom{r+2-t}{s-1}\binom{n-r-2}{r-s}+\sum_{i=r+3}^{2s-1}\binom{i-t-1}{s-2}\binom{n-i}{r-s}$
only appears for $2s-1\ge r+2$.
\end{lem}
\begin{proof}
We begin by counting the cases with $s,t\in\{2,3\}$. All cases use
the same method, so we illustrate the $t=2$, $s=3$ case, so $X=\{4,r+2\}$.
There are $10$ possibilities for the first $3$ elements in $A_{n,r,3}$
and we separate them into groups $A,B,C$, and $D$. Sets of type
$A$ end in $3$. Sets of type $B$ end in $4$. Sets of type $C$
end in $5$, and don't contain $4$. Sets of type $D$ end in $5$
and contain $4$.
\begin{eqnarray*}
 & A:\{1,2,3\},\, B:\{1,2,4\},\{1,3,4\},\{2,3,4\}\\
 & C:\{1,2,5\},\{1,3,5\},\{2,3,5\},\, D:\{1,4,5\},\{2,4,5\},\{3,4,5\}
\end{eqnarray*}
For elements of type $A$, a set starting with $\{1,2,3\}$ is in
$\A_{n,r,3}(X)$ if $4$ or $r+2$ appear as one of the other elements,
so we use inclusion-exclusion to count $2\binom{n-4}{r-4}-\binom{n-5}{r-5}$
elements of $\A_{n,r,3}(X)$ with first $3$ elements $\{1,2,3\}$.
Since $4$ is in each set of type $B$, all of the $3\binom{n-4}{r-3}$
elements of $\A_{n,r,3}$ that have a set of type $B$ as the first
$3$ elements are in $\A_{n,r,3}(X)$. If the first 3 elements of
a set are of type $C$, then it cannot contain $4$, so there are
$3\binom{n-6}{r-4}$ sets in $\A_{n,r,3}$ with the first $3$ elements
of type $C$. Similar to sets of type $B$, there are $3\binom{n-5}{r-3}$
elements of $\A_{n,r,3}(X)$ with first $3$ elements of type $D$.
We leave it to the reader to check the formulas in the other cases.

We now count the general case. The method we use cannot be applied
in the case that both $s$ and $t$ are in $\{2,3\}$. However, we
will see that in the case $s=3$, the two methods give the same formula.

We first count the number of elements of $\A_{n,r,s}$ that contain
some element of $X$ other than $r+2$. We claim that there are 
\begin{equation}
\sum_{i=s}^{2s-1}\left(\binom{i-1}{s-1}-\binom{i-t}{s-1}\right)\binom{n-i}{r-s}\label{eq:count1}
\end{equation}
such elements. We begin by considering the case $t\ge4$. We consider
elements of $\A_{n,r,s}$ with $s$-th element $i$, for $s\le i\le2s-1$.
There are a total of $\binom{i-1}{s-1}\binom{n-i}{r-s}$ elements
with $s$-th element $i$, because the $s-1$ elements in positions
$1,\ldots,s-1$ are chosen from $[i-1]$, and the $r-s$ elements
in positions $i+1,\ldots,r$ are chosen from $[i+1,n]$. We now count
the number of elements with $s$-th element $i$ that do not contain
an element of $[2,t]$. We claim this number is $\binom{i-t}{s-1}\binom{n-i}{r-s}$.
Notice that if $i\in[2,t]$, then $i-t<s-1$, so in this case $\binom{i-t}{s-1}=0$
and the claim holds. If none of $2,\ldots,t$ occur in the first $s$
positions of some $B\in\A_{n,r,s}$, then $B\cap[2,t]=\emptyset$,
because in this case the second element must be at least $t+1$. This
means that if $B\cap[2,t]\ne\emptyset$, then $B$ contains an element
of $[2,t]$ in the first $s$ positions, so we just need to count
the number of sets with $i$ as the $s$-th element, that do not contain
an element of $[2,t]$ in the first $s-1$ positions, which gives
$\binom{i-t}{2-1}\binom{n-i}{r-s}$. To count the number of elements
of $\A_{n,r,s}$ that contain an element of $X$, we subtract the
number that do not intersect $X$ from the total, to get \eqref{eq:count1}
in the case $t\ge4$. 

In the case $t=1$, we have $\binom{i-1}{s-1}-\binom{i-t}{s-1}=0$.
In this case, the only element of $X$ is $r+2$, so \eqref{eq:count1}
trivially counts the number of elements of $\A_{n,r,s}$ containing
an element of $X$ that is not $r+2$. In the case $s,t\in\{2,3\}$,
the above argument may fail in that ``if $B\cap[2,t]\ne\emptyset$,
then $B$ contains an element of $[2,t]$ in the first $s$ positions''
is no longer applicable, since in these cases, $X\backslash\{r+2\}=\{4\}$
or $\{2,4\}$, which is not of the form $[2,t]$. For example, if
$s=2$, there exist sets with first 3 elements $1,2,4$, which contains
the element $4\in X$ in the last $r-s$ elements. However, for $s\ge4$,
then we still have ``if $B\cap(X\backslash\{r+2\})\ne\emptyset$,
then $B$ contains an element of $(X\backslash\{r+2\})$ in the first
$s$ positions'' so the same argument of the previous paragraph holds
in the case $t\in\{2,3\}$ and $s\ge4$.

We now count the number of elements containing $r+2$, but contain
no other elements of $X$. We first count the number of elements of
$\A_{n,r,s}$ that contain $r+2$ among the last $r-s$ elements.
We claim there are 
\begin{equation}
\sum_{i=s}^{\min\{r+1,2s-1\}}\binom{i-t}{s-1}\binom{n-i-1}{r-s-1}\label{eq:count2}
\end{equation}
such elements. As before, we consider sets with $i$ as the $s$-th
element, except in this case we restrict $s\le i\le\min\{r+1,2s-1\}$.
There are $\binom{i-t}{s-1}$ ways to choose the first $s-1$ elements
so as not to contain any element of $X\backslash\{r+2\}$ in the first
$s-1$ elements, and then there are $\binom{n-i-1}{r-s-1}$ ways to
choose the remaining $r-s$ elements so that $r+2$ is one of them.
Summing over $i$ gives \eqref{eq:count2}.

We now count the number of elements of $\A_{n,r,s}$ that contain
$r+2$ in the first $s$ elements, but contain no other elements of
$X$. We claim there are 
\begin{equation}
\binom{r+2-t}{s-1}\binom{n-r-2}{r-s}+\sum_{i=r+3}^{2s-1}\binom{i-t-1}{s-2}\binom{n-i}{r-s}\label{eq:count3}
\end{equation}
such elements, assuming $2s-1\ge r+2$. There are $\binom{r+2-t}{s-1}\binom{n-r-2}{r-s}$
elements with $r+2$ as the $s$-th element. As before, we count the
number of sets with $i$ as the $s$-th element. There are $\binom{i-t-1}{s-2}$
possibilities for the first $s$ elements (because $i$ and $r+2$
must be among them), and $\binom{n-i}{r-s}$ possibilities for the
remaining elements. This gives \eqref{eq:count3}.

Adding \eqref{eq:count1}, \eqref{eq:count2}, \eqref{eq:count3} together,
we obtain \eqref{eq:countform}.

It is important to notice that in the case $s=3$ and $t\in\{2,3\}$,
even though we cannot use the same argument, we still obtain the same
formula, i.e. we have 
\[
|\A_{n,r,3}(X)|=\sum_{i=s}^{2s-1}\left(\binom{i-1}{s-1}-\binom{i-t}{s-1}\right)\binom{n-i}{r-s}+\sum_{i=s}^{\min\{r+1,2s-1\}}\binom{i-t}{s-1}\binom{n-i-1}{r-s-1}
\]
for $s=3$. (Notice the other terms of \eqref{eq:countform} do not
appear since for $r\ge4$ and $s=3$, we have $2s-1<r+2$). The proof
of this is simple. When $t=3$, the binomial coefficients match exactly.
When $t=2$, we must use the relation $\binom{n}{k}+\binom{n}{k+1}=\binom{n+1}{k+1}$,
but it is still straightforward. \end{proof}

\begin{rem} Lemma \ref{l.countform} gives $|\A_{n,r,s}(X)|$ for
$r\ge4$. In the case $r=3$, it is easy to count $|\A_{n,3,2}(\{5\})|=3$,
$|\A_{n,3,2}(\{4,5\})|=6$, $|\A_{n,3,2}(\{2,4,5\})|=2n-3$, $|\A_{n,3,3}(\{5\})|=6$,
$|\A_{n,3,3}(\{4,5\})|=9$, $|\A_{n,3,3}(\{2,4,5\})|=10$.
\end{rem}

In the next two subsections, we analyze \eqref{eq:countform} term
by term. 
\subsection{\eqref{eq:count3} is negligibly small for large $r$\label{ss.sink}}

In this section, we show that \eqref{eq:count3} goes to $0$ exponentially
in $r$.

\begin{lem}\label{l.goestozero}

Let $n/r=a$, and $r\ge4$. Let 
\[
Q(n,r,s,t)=\begin{cases}
\frac{\binom{r+2-t}{s-1}\binom{n-r-2}{r-s}+\sum_{i=r+3}^{2s-1}\binom{i-t-1}{s-2}\binom{n-i}{r-s}}{\binom{n-2}{r-2}} & 2s-1\ge r+2\\
0 & else
\end{cases}
\]
For any fixed $a\ge2.1$, $\max\{Q(n,r,s,t)\}_{s=2}^{r}$ goes to
$0$ exponentially in $r$ as $r\rightarrow\infty$, for any $s$.
For $a\ge\varphi^{2}$, we have $Q(ar,r,s,t)\le3.25r^{3/2}(.954)^{r}$.

\end{lem}

\begin{proof}
We address each term in ${\displaystyle \sum_{i=r+3}^{2s-1}\binom{i-t-1}{s-2}\binom{n-i}{r-s}}$
individually. We begin by using Stirling's approximation, and then
use a computer to find the maximal term. 

Notice that in the case $r=s$, the $Q(n,r,s,t)$ reduces to 
\[
\frac{\binom{r+2-t}{s-1}+\sum_{i=r+3}^{2s-1}\binom{i-t-1}{s-2}}{\binom{n-2}{r-2}}.
\]
Since the numerator is independent of $n$, the result is simple. 

We now address the case $2\le s<r$. Notice that ${\displaystyle \binom{i-t-1}{s-2}-\binom{i-(t-1)-1}{s-2}=\binom{i-t-2}{s-3}\ge0}$,
so $\binom{i-t-1}{s-2}$ is decreasing in $t$, so we may assume $t=1$.
We first address ${\displaystyle \frac{\sum_{i=r+3}^{2s-1}\binom{i-t-1}{s-2}\binom{n-i}{r-s}}{\binom{n-2}{r-2}}}$.
The same method applies for ${\displaystyle \binom{r+2-t}{s-1}\binom{n-r-2}{r-2}/\binom{n-2}{r-2}}$,
so we omit the analysis of this term.

Define $q(n,r,i,s)$ by
\[
q(n,r,i,s)=\frac{\binom{i-2}{s-2}\binom{n-i}{r-s}}{\binom{n-2}{r-2}}=\frac{s(s-1)n(n-1)}{i(i-1)r(r-1)}\cdot\frac{i!(n-i)!r!(n-r)!}{s!(i-s)!(r-s)!(n-i-r+s)!n!}.
\]
Recall Stirling's approximation $n!\sim\sqrt{2\pi n}\left(n/e\right)^{n}$.
Although this only holds for large $n$, in general we have $\sqrt{2\pi n}\left(n/e\right)^{n}\le n!\le e\sqrt{n}\left(n/e\right)^{n}$.
Applying this gives
\begin{eqnarray*}
q(n,r,i,s) & \le & \frac{s(s-1)n(n-1)e^{4}}{i(i-1)r(r-1)(2\pi)^{2}}\cdot\sqrt{\frac{i(n-i)r(n-r)}{2\pi s(i-s)(r-s)(n-i-r+s)n}}\\
 &  & \cdot\frac{i^{i}(n-i)^{n-i}r^{r}(n-r)^{n-r}}{s^{s}(i-s)^{i-s}(r-s)^{r-s}(n-i-r+s)^{n-i-r+s}n^{n}}
\end{eqnarray*}

Notice many of the powers of $e$ from Stirling's approximation have
canceled. We substitute $n=ar$. We wish to find for which $a$ this
term goes to $0$ exponentially in $r$. To do this, divide numerator
and denominator by $r^{2ar}$, and pull out an $r$-th root. We also
substitute $I=i/r$, and $S=s/r$. We have also used the fact $(s-1)/(i-1)\le s/i$.
This gives
\begin{eqnarray}
q(n,r,i,s)\le T(n,r,i,s) & = & \frac{s^{2}a(ar-1)e^{4}}{i^{2}(r-1)(2\pi)^{2}}\cdot\sqrt{\frac{i(ar-i)(ar-r)}{2\pi s(i-s)(r-s)(ar-i-r+s)a}}\nonumber \\
 &  & \cdot\left(\frac{I^{I}(a-I)^{a-I}(a-1)^{a-1}}{S^{S}(I-S)^{I-S}(1-S)^{1-S}(a-I-1+S)^{a-I-1+S}a^{a}}\right)^{r}.\label{eq:exppart}
\end{eqnarray}

We first address ${\displaystyle B(I,S,a)=\frac{I^{I}(a-I)^{a-I}(a-1)^{a-1}}{S^{S}(I-S)^{I-S}(1-S)^{1-S}(a-I-1+S)^{a-I-1+S}a^{a}}}$.
Notice that we have the bounds $1/2\le S\le1$ and $1\le I\le2S$.
Using a computer, we compute $\frac{\partial}{\partial a}B(I,S,a)$,
and then maximize this function. We find it is negative for all $I$
and $S$ in the previously stated range, with $a\ge2$. Thus for fixed
$I$ and $S$, we have $B(I,S,a)$ is decreasing in $a$. We find
that the maximum occurs at $S=\frac{1}{2}$, $I=1$. This gives $B(1,\frac{1}{2},2.1)=.9978$.
Thus each term in the first sum of \eqref{eq:countform} goes to $0$
exponentially in $r$ for $a\ge2.1$. To finish the proof of the first
part of the proposition, notice that we have shown each term in the
sum goes to $0$ exponentially in $r$. Since there are at most $r-4$
terms, the sum still goes to $0$ exponentially in $r$. 

We now wish to evaluate how quickly this term goes to $0$ when $a\ge\varphi^{2}$.
Using a computer as above, we find $T(n,r,i,s)$ is decreasing in
$a$, so $q(n,r,i,s)\le T(\varphi^{2}r,r,i,s)$ for $a\ge\varphi^{2}$.
Maximizing this, we find $T(\varphi^{2}r,r,i,s)\le3.25\sqrt{r}(.954)^{r}$.
Accounting for the fact that there are less than $r$ terms, we get
$Q(ar,r,s,t)\le3.25r^{3/2}(.954)^{r}$ for $a\ge\varphi^{2}$. \end{proof}

\subsection{\eqref{eq:count1}+\eqref{eq:count2} is decreasing in $s$ \label{ss.decins}}

In this section, we first show that \eqref{eq:count1} is decreasing
in $s$. It is not true that \eqref{eq:count2} is always decreasing
in $s$, but when combined with \eqref{eq:count1}, the total is decreasing
in $s$. To show this, we must address several cases. We address \eqref{eq:count1}
first. 

\begin{lem}\label{l.decinswt}

For fixed $n,r$, such that $n/r=a\ge2$, we have $g(n,r,s,t)=\sum_{i=s}^{2s-1}\left(\binom{i-1}{s-1}-\binom{i-t}{s-1}\right)\binom{n-i}{r-s}$
is decreasing in $s$.

\end{lem}

\begin{proof}

The idea of this proof is to use summation by parts to get an expression
for $g(n,r,s,t)$ in terms of $g(n,r,s+1,t)$, along with some error
terms. We then analyze the error terms to find for which $a$ we have
$g(n,r,s,t)\ge g(n,r,s+1,t)$.

We begin by applying summation by parts
\[
\sum_{i=s}^{2s-1}\binom{i-t}{s-1}\binom{n-i}{r-s}=\binom{n-2s+1}{r-s}\sum_{i=s}^{2s-1}\binom{i-t}{s-1}-\sum_{i=s}^{2s-2}\left(\binom{n-i-1}{r-s}-\binom{n-i}{r-s}\right)\sum_{k=s}^{i}\binom{k-t}{s-1}
\]
We use the identities $\sum_{k=s}^{i}\binom{k-t}{s-1}=\binom{i-t+1}{s}$
and $\binom{n-i-1}{r-s}-\binom{n-i}{r-s}=-\binom{n-i-1}{r-s-1}$.
We also change the indices on the outer sum in the second term, which
accounts for the $\binom{n-2s+1}{r-s}\binom{2s-t}{s}-\binom{n-2s-1}{r-s-1}\binom{2s+1-t}{s}-\binom{n-2s}{r-s-1}\binom{2s-t}{s}$
below 
\begin{eqnarray}
\sum_{i=s}^{2s-1}\binom{i-t}{s-1}\binom{n-i}{r-s} & = & \sum_{i=s+1}^{2s+1}\binom{n-i}{r-s-1}\binom{i-t}{s}+\binom{n-2s+1}{r-s}\binom{2s-t}{s}\nonumber \\
 &  & -\binom{n-2s-1}{r-s-1}\binom{2s+1-t}{s}-\binom{n-2s}{r-s-1}\binom{2s-t}{s}.\label{eq:sumbyparts}
\end{eqnarray}
Subtracting equation \eqref{eq:sumbyparts} from equation \eqref{eq:sumbyparts}
with the substitution $t=1$ gives 
\begin{eqnarray*}
g(n,r,s,t) & = & g(n,r,s+1,t)+\binom{n-2s+1}{r-s}\left(\binom{2s-1}{s}-\binom{2s-t}{s}\right)-\\
 &  & \binom{n-2s-1}{r-s-1}\left(\binom{2s}{s}-\binom{2s+1-t}{s}\right)-\binom{n-2s}{r-s-1}\left(\binom{2s-1}{s}-\binom{2s-t}{s}\right).
\end{eqnarray*}

Let $A=\binom{2s}{s}$, $B=\binom{2s-t}{s}$, and $C=\binom{2s+1-t}{s}$.
Notice $\binom{2s}{s}=2\binom{2s-1}{s}$. By expanding the binomial
coefficients, and pulling out common terms, we find
\begin{eqnarray*}
g(n,r,s,t)-g(n,r,s+1,t) & = & \frac{(n-2s-1)!}{2(r-s)!(n-r-s+1)!}((n-2s+1)(n-2s)(A-2B)\\
 &  & -2(r-s)(n-r-s+1)(A-C)-(n-2s)(r-s)(A-2B)).
\end{eqnarray*}

We want to find when $g(n,r,s,t)-g(n,r,s+1,t)\ge0$. We substitute
$n=ar$, and multiply by $\frac{2(r-s)!(n-r-s+1)!}{(n-2s-1)!}$. This
gives us a quadratic expression in $a$. By finding the roots, we
will find for which $a$ the function $g(n,r,s,t)$ is decreasing
in $s$. Using the quadratic formula to find the roots, we see that
the discriminant is a square, namely
\[
\left(r^{2}A+2r^{2}B+rA-2rB-rsA-2rsB-2r^{2}C+2rsC\right)^{2}.
\]
This gives that the roots are $a=1+s/r-1/r$ and $a=1+s/r$. This
means $g(n,r,s,t)\ge g(n,r,s+1,t)$ for $n/r=a\ge1+s/r$. Since $s/r\le1$,
we have $g(n,r,s,t)$ is decreasing in $s$ for $a\ge2$.\end{proof}

We may now show \eqref{eq:count1}+\eqref{eq:count2} is decreasing
in $s$. 

\begin{lem}\label{l.ynofactor}

For fixed $n,r$, if $n/r=a>7/3$, then
\[
D(n,r,s,t)=\sum_{i=s}^{2s-1}\left(\binom{i-1}{s-1}-\binom{i-t}{s-1}\right)\binom{n-i}{r-s}+\sum_{i=s}^{\min\{r+1,2s-1\}}\binom{i-t}{s-1}\binom{n-i-1}{r-s-1}
\]
 is decreasing in $s$.

\end{lem}

\begin{proof}

We first consider $s$ such that $\min\{2s-1,r+1\}=r+1$, which gives
$s\ge r/2+1$. In this case, when $s$ increases by 1, the expression
${\displaystyle \sum_{i=s}^{\min\{r+1,2s-1\}}\binom{i-t}{s-1}\binom{n-i-1}{r-s-1}}$
loses a term, and each term is smaller, so this sum is decreasing
in $s$. By Lemma \ref{l.decinswt}, we have $g(n,r,s,t)=\sum_{i=s}^{2s-1}\left(\binom{i-1}{s-1}-\binom{i-t}{s-1}\right)\binom{n-i}{r-s}$
is decreasing in $s$, so $D(n,r,s,t)$ is decreasing in $s$ for
$s\ge r/2+1$. 

We now may assume $\min\{2s-1,r+1\}=2s-1$. In this case, we have
\[
D(n,r,s,t)=\sum_{i=s}^{2s-1}\binom{i-t}{s-1}\binom{n-i-1}{r-s-1}+\left(\binom{i-1}{s-1}-\binom{i-t}{s-1}\right)\binom{n-i}{r-s}.
\]

We must address the cases $t=1$ and $t\ge2$ separately, because
when $t=1$, we have $\binom{i-1}{s-1}-\binom{i-t}{s-1}=0$. The case
$t=1$ may be proved using a proof similar to that of Lemma \ref{l.decinswt},
although it is much simpler. Thus we just consider the case $t\ge2$.

Let $A=\binom{2s-t}{s}$, $B=\binom{2s+1-t}{s}$, and $C=\binom{2s}{s}$.
We begin by simplifying the expression. By the identity $\binom{n}{k}+\binom{n}{k+1}=\binom{n+1}{k+1}$,
we have
\[
\binom{i-t}{s-1}\binom{n-i-1}{r-s-1}+\left(\binom{i-1}{s-1}-\binom{i-t}{s-1}\right)\binom{n-i}{r-s}=\binom{i-1}{s-1}\binom{n-i}{r-s}-\binom{i-t}{s-1}\binom{n-i-1}{r-s}.
\]
As in Lemma \ref{l.decinswt}, we use summation by parts. We use \eqref{eq:sumbyparts}
twice, once with $t=1$, and once with general $t$, but $n$ and
$r$ replaced by $n-1$ and $r-1$ to get 
\begin{eqnarray*}
D(n,r,s,t) & = & \binom{n-2s+1}{r-s}\frac{C}{2}-\binom{n-2s-1}{r-s-1}C-\binom{n-2s}{r-s-1}\frac{C}{2}\\
 &  & -\binom{n-2s}{r-s}A+\binom{n-2s-2}{r-s-1}B+\binom{n-2s-1}{r-s-1}A+D(n,r,s+1,t).
\end{eqnarray*}
We expand the binomial coefficients in the above expression to get
$D(n,r,s,t)-D(n,r,s+1,t)$ is equal to
\begin{eqnarray*}
 &  & \frac{(n-2s-2)!}{2(n-r-s+1)!(r-s)!}\left((n-2s+1)(n-2s)(n-2s-1)C-\right.\\
\\
 &  & 2(n-2s)(n-2s-1)(n-r-s+1)A+2(r-s)(n-r-s)(n-r-s+1)B\\
 &  & +2(n-2s-1)(r-s)(n-r-s+1)A).
\end{eqnarray*}

We wish to find for which $a$ we have $D(n,r,s,t)-D(n,r,s+1,t)\ge0$.
By substituting the above expression into this inequality, multiplying
by $2(r-s)!(n-r-s+1)!/(n-2s-2)!$, and substituting $n=ar$, we get
the following inequality, cubic in $a$,
\begin{eqnarray*}
0 & \le & (-2r^{3}A+r^{3}C)a^{3}+(4r^{3}A+8r^{2}sA+2r^{3}B-2r^{2}sB-3r^{3}C-3r^{2}sC)a^{2}+\\
 &  & (-4r^{3}B+4rs^{2}B+2r^{3}C+8r^{2}sC+2rs^{2}C+2r^{2}B-2rsB+r^{2}C-rsC+2rA-rC)a\\
 &  & +4r^{2}sA+8rs^{2}A+4s^{3}A+2r^{3}B+2r^{2}sB-2rs^{2}B-2s^{3}B-4r^{2}sC-4rs^{2}C-2r^{2}B\\
 &  & +2s^{2}B-2r^{2}C+2rsC-4sA+2rC.
\end{eqnarray*}

We wish to find the roots of the left hand side, so that we will know
for which $a$ we have $D(n,r,s,t)-D(n,r,s+1,t)\ge0$. Recall that
in the proof of Lemma \ref{l.decinswt}, we saw $a=1+s/r-1/r$ was
a root. Using a computer, it is easy to check that $(a-(1+s/r-1/r))$
is a factor of the cubic equation as well. Thus we have reduced the
problem to finding the roots of the quadratic
\begin{eqnarray}
0 & = & (-2Ar^{3}+Cr^{3})a^{2}+(2Ar^{3}+2Br^{3}-2Cr^{3}+6Ar^{2}s-2Br^{2}s-2Cr^{2}s+2Ar^{2}-Cr^{2})a\nonumber \\
 &  & -2Br^{3}-4Ar^{2}s+4Cr^{2}s-4Ars^{2}+2Brs^{2}-2Ar^{2}+2Cr^{2}-2Ars.\label{eq:quada}
\end{eqnarray}
Unfortunately, the discriminant of the remaining quadratic equation
is not a square, as in the case of the previous lemma. However, we
do have that the discriminant is 
\[
\Delta=(2Ar^{3}-2Ar^{2}+Cr^{2}-2Br^{3}-2Cr^{3}+2Br^{2}s-2Ar^{2}s+2Cr^{2}s)^{2}-8BCr^{4}(r+s)^{2}.
\]
Let 
\[
\Delta'=2Ar^{3}-2Ar^{2}+Cr^{2}-2Br^{3}-2Cr^{3}+2Br^{2}s-2Ar^{2}s+2Cr^{2}s.
\]
Notice the roots of \eqref{eq:quada} are of the form $\frac{b\pm\sqrt{\Delta}}{c}$.
Also observe that $(\Delta')^{2}>\Delta$, so $\frac{b+\sqrt{(\Delta')^{2}}}{c}>\frac{b+\sqrt{\Delta}}{c}$
(notice $c=2(C-2A)r^{3}>0$), hence we obtain a bound on the larger
root. It is not immediately obvious which is the larger root since
it is not trivial to know if $\Delta'>0$, thus we must consider both
$\frac{b\pm\Delta'}{c}$, which are
\[
\frac{-4Br^{3}-8Ar^{2}s+4Br^{2}s+4Cr^{2}s-4Ar^{2}+2Cr^{2}}{2(C-2A)r^{3}},\frac{-4Ar^{3}+4Cr^{3}-4Ar^{2}s}{2(C-2A)r^{3}}.
\]
For the first root, notice $\frac{-4Ar^{2}+2Cr^{2}}{2(C-2A)r^{3}}=\frac{1}{r}$
and $\frac{-8Ar^{2}s+4Cr^{2}s}{2(C-2A)r^{3}}=\frac{2s}{r}$ and $\frac{-4Br^{3}+4Br^{2}s}{2(C-2A)r^{3}}=\frac{2B}{C-2A}(\frac{s}{r}-1)$.
This gives
\[
\frac{-4Br^{3}-8Ar^{2}s+4Br^{2}s+4Cr^{2}s-4Ar^{2}+2Cr^{2}}{2(C-2A)r^{3}}=\frac{2s}{r}+\frac{1}{r}+\frac{2B}{(C-2A)}\left(\frac{s}{r}-1\right).
\]
This will be maximized when $\frac{s}{r}$ is maximized, hence is
at most $2+1/r\le2+1/4$ since $r\ge4$.

For the second root, observe
\[
\frac{-4Ar^{3}+4Cr^{3}-4Ar^{2}s}{2(C-2A)r^{3}}=2+\frac{2A}{C-2A}\left(1-\frac{s}{r}\right).
\]
If $t\ge s$, then $A=0$, so the above expression is just $2$. In
the case $t<s$, notice that for fixed $s$ $A$ is decreasing in
$t$, so $\frac{2A}{C-2A}$ is decreasing in $t$ as well. We first
address the case $t=3$. In this case, 
\[
\frac{2A}{C-2A}=\frac{2}{\frac{2s(2s-1)(2s-2}{s(s-1)(s-2)}-2}=\frac{s-2}{3s}=\frac{1}{3}-\frac{2}{s}.
\]
This gives
\[
2+\frac{2A}{C-2A}\left(1-\frac{s}{r}\right)\le2+\left(\frac{1}{3}-\frac{2}{s}\right)\left(1-\frac{s}{r}\right)<2+\frac{1}{3}
\]
for $t\ge3$. 

In the case $t=2$, we do not get the desired bound. This means the
approximation made for the discriminant is too weak in this case,
so we must re-evaluate the roots of the quadratic \eqref{eq:quada}
in $a$, with better approximations. Since $t=2$, we have $A=\frac{s-1}{2(2s-1)}C$
and $B=\frac{C}{2}$. When we make these substitutions, the quadratic
\eqref{eq:quada} simplifies so that it is divisible by $C$. We divide
by $C$, and multiply by $2s-1$ to clear denominators to get the
quadratic
\[
r^{3}sa^{2}+(-r^{3}s-3r^{2}s^{2}-r^{2}s)a-2r^{3}s+6r^{2}s^{2}+r^{3}+r^{2}s-r^{2}+rs=0.
\]
The roots of this quadratic are
\[
\frac{rs+3s^{2}+s\pm\sqrt{9r^{2}s^{2}-18rs^{3}+9s^{4}-4r^{2}s-2rs^{2}+6s^{3}+4rs-3s^{2}}}{2rs}.
\]
Using a computer to maximize these roots, we see that both are at
most $2$. Thus we have bounded all roots by $7/3$, which proves
the lemma. \end{proof}

\subsection{Other Preliminaries}

We must now count $|\S(X)|$.

\begin{lem}\label{l.scount}

Let $X\subset[2,n]$, with $|X|=t$. We have
\[
|\S(X)|=\binom{n-1}{r-1}-\binom{n-1-t}{r-1}.
\]
\end{lem}

\begin{proof}

There are $\binom{n-1}{r-1}$ elements of $\S$, and there are $\binom{n-1-t}{r-1}$
ways to choose $r-1$ elements from $[2,n]\backslash X$. Subtracting
them gives the desired formula.\end{proof}

In what follows, it is simplest to evaluate $|\A_{n,r,s}(X)|/|\S(X)|$
using the fact
\[
\frac{|\A_{n,r,s}(X)|}{|\S(X)|}=\frac{|\A_{n,r,s}(X)|}{\binom{n-2}{r-2}}\cdot\frac{\binom{n-2}{r-2}}{|\S(X)|}.
\]
To evaluate this, the following lemma will be useful.

\begin{lem}\label{l.bincoflim}

For constants $b,c\in\mathbb{N}$, with $b\ge c$ and $n=ar$, we
have 
\[
\lim_{r\rightarrow\infty}\frac{\binom{n-b}{r-c}}{\binom{n-2}{r-2}}=\frac{(a-1)^{b-c}}{a^{b-2}}.
\]

\end{lem}

\begin{proof}

We expand the binomial coefficients. 
\[
\frac{\binom{n-b}{r-c}}{\binom{n-2}{r-2}}=\frac{(r-2)\cdots(r-c+1)\cdot(n-r)\cdots(n-r-b+c+1)}{(n-2)\cdots(n-b+1)}.
\]
For large $r$, the constants are negligible in comparison with $n,r$,
so this is approximately
\[
\frac{r^{c-2}(n-r)^{b-c}}{n^{b-2}}=\frac{(a-1)^{b-c}}{a^{b-2}},
\]
as desired.\end{proof}

\begin{lem}\label{l.lastpart}

For $r$ sufficiently large, we have 
\[
\frac{\binom{n-2}{r-2}}{|\S(X)|}\le\frac{1}{a-a(1-\frac{1}{a})^{t}}.
\]
Furthermore, for fixed $r$, $\binom{n-2}{r-2}/|\S(X)|$ decreases
to $1/t$ as $n\rightarrow\infty$.

\end{lem}

\begin{proof}

We expand to get 
\[
\frac{\binom{n-2}{r-2}}{|\S(X)|}=\frac{(r-1)(n-2)\cdots(n-t)}{(n-1)\cdots(n-t)-(n-r)\cdots(n-t-r+1)}
\]
\[
=\frac{r-1}{(n-1)\left(1-\frac{(n-r)\cdots(n-t-r+1)}{(n-1)\cdots(n-t)}\right)}\le\frac{r-1}{(n-1)\left(1-\left(1-\frac{r-1}{n-1}\right)^{t}\right)}.
\]
Let $x=\frac{r-1}{n-1}$, so the above expression is $\frac{x}{1-(1-x)^{t}}$.
Notice that for large $r$ we have $x\approx1/a$, which proves
part of the lemma. Observe $\frac{x}{1-(1-x)^{t}}\rightarrow\frac{1}{t}$
as $x\rightarrow0$, by L'Hopital's rule. By taking derivatives it
is easy to check that $\frac{x}{1-(1-x)^{t}}$ decreases as $x$ decreases
to $0$. Since $x=\frac{r-1}{n-1}$ is decreasing in $n$, we have
$\frac{\frac{r-1}{n-1}}{1-(1-\frac{r-1}{n-1})^{t}}$ decreases to
$1/t$ as $n$ increases. \end{proof}%
\begin{comment}
\begin{cor}\label{c.swbincof}

For $r$ sufficiently large in comparison with $t$, we have 

\[
\frac{|\S(X)|}{\binom{n-2}{r-2}}\approx a-a\left(\frac{a-1}{a}\right)^{t}.
\]

\end{cor}

\begin{proof}

By Lemmas \ref{l.scount} and \ref{l.bincoflim}

\[
\frac{|\S(X)|}{\binom{n-2}{r-2}}=\frac{\binom{n-1}{r-1}-\binom{n-1-t}{r-1}}{\binom{n-2}{r-2}}\approx a-a\left(\frac{a-1}{a}\right)^{t}
\]
\end{proof}
\end{comment}

\section{Proof of Theorem \ref{t.onegen} \label{s.onegen}}

We may now complete the proof of Theorem \ref{t.onegen}, which we
restate for easy reference.

\setcounter{count}{4}

\begin{thm} Let $r\ge3$, let $\A=F(r,n,\G)$ be an intersecting
family with $|\G|=1$, and let $X$ be eventually EKR. Then $n>\varphi^{2}r$
implies $|\A(X)|\le|\S(X)|$. \end{thm}

\begin{proof}The proof breaks into several cases, depending on $t$.
Each case uses the the same method that we outline, with justification,
in the following paragraph. Having done so we may omit details in
the various cases by referring to the following paragraph.

By Lemma \ref{l.countform}, we have a general formula for $|\A_{n,r,s}(X)|$
for any $s$ and $X\in\{\{r+2\},\{4,r+2\},\{2,4,r+2\},\{2,\ldots,t,r+2\}\}$
that works in all but a few exceptional cases. For each $X$, we begin
by addressing the exceptional cases, which happen when $s=2$ and
$t\ge2$ (although Lemma \ref{l.countform} does not give an exceptional
case in the case $t\ge4$, it is best to address the case $s=2$,
$t\ge4$ separately from the other possible values of $s$ when $t\ge4$).
These will be dealt with on a case by case basis. We will then address
the general formula for $|\A_{n,r,s}(X)|$. We have seen that it breaks
into the sum of two parts. The smaller part, $Q(n,r,s,t)$ of Lemma
\ref{l.goestozero}, is negligible for large $r$. In particular,
by Lemma \ref{l.goestozero} we have $Q(n,r,s,t)/\binom{n-2}{r-2}\le3.25r^{3/2}(.954)^{r}$.
The other term 
\[
D(n,r,s,t)=\sum_{i=s}^{2s-1}\left(\binom{i-1}{s-1}-\binom{i-t}{s-1}\right)\binom{n-i}{r-s}+\sum_{i=s}^{\min\{r+1,2s-1\}}\binom{i-t}{s-1}\binom{n-i-1}{r-s-1}
\]
is decreasing in $s$ by Lemma \ref{l.ynofactor}, so it suffices
to evaluate $D(n,r,s_{0},t)$ where $s_{0}=2$ for $t=1$ and $s_{0}=3$
for $t\ge2$. These values of $s_{0}$ come from the exceptional cases
mentioned above. Thus we have 
\begin{equation}
\frac{|\A_{n,r,s}(X)|}{|\S(X)|}=\frac{|\A_{n,r,s}(X)|}{|\binom{n-2}{r-2}}\cdot\frac{\binom{n-2}{r-2}}{|\S(X)|}\le\left(\frac{D(n,r,s_{0},t)}{\binom{n-2}{r-2}}+3.25r^{3/2}(.954)^{r}\right)\frac{\binom{n-2}{r-2}}{|\S(X)|}\label{eq:genidea}
\end{equation}
for $s\ge s_{0}$. We will then find an $r_{0}$ such that for $r\ge r_{0}$,
we have $n>\varphi^{2}r$ implies that the right hand side of \eqref{eq:genidea}
is less than 1. We will do this by showing \eqref{eq:genidea} is
decreasing in $a:=n/r$, then substituting $a=\varphi^{2}$ and finding
$r_{0}$. We then use a computer to check the same result holds for
$3\le r<r_{0}$. 

We now apply this method for specific values of $s$ and $t$.\setcounter{count}{19}

\paragraph*{Case 1: $t=1$}

As discussed above, we have
\begin{equation}
\frac{|\A_{n,r,s}(X)|}{|\S(X)|}\le\frac{\sum_{i=2}^{3}\binom{i-1}{1}\binom{n-i-1}{r-3}}{\binom{n-2}{r-2}}+3.25r^{3/2}(.954)^{r}.\label{eq:xeq1mess}
\end{equation}
Notice the fraction in \eqref{eq:xeq1mess} is

\[
q(n,r):=\frac{\binom{n-3}{r-3}+2\binom{n-4}{r-3}}{\binom{n-2}{r-2}}=\frac{(r-2)}{(n-2)}+2\frac{(r-2)(n-r)}{(n-2)(n-3)}.
\]
Using the fact that for $b\le c$, we have $\frac{b}{c}\le\frac{b+1}{c+1}$,
we have
\[
\frac{(r-2)}{(n-2)}+2\frac{(r-2)(n-r)}{(n-2)(n-3)}\le\frac{r}{n}+\frac{2r}{n}\left(\frac{n-r+3}{n}\right)=\frac{1}{a}+\frac{2}{a}\left(1-\frac{1}{a}+\frac{3}{ar}\right).
\]
It is easy to check that this is decreasing in $a$, so it suffices
to assume $a=\varphi^{2}$. It is clearly decreasing in $r$, so we
may also assume that $r\ge200$. Substituting $r=200$ and $a=\varphi^{2}$,
we have 
\[
\frac{1}{a}+\frac{2}{a}\left(1-\frac{1}{a}+\frac{3}{ar}\right)\le.859.
\]
Thus $r\ge200$, we have $\frac{|\A_{n,r,s}(X)|}{|\S(X)|}\le.859+3.25r^{3/2}(.954)^{r}$.
Using a computer to evaluate $3.25r^{3/2}(.954)^{r}$, we get that
$\frac{|\A_{n,r,s}(X)|}{|\S(X)|}<1$ for $r_{0}=242$, and hence for
$r\ge242$. 

It remains to check the cases where $r\le241$. Recall
\begin{eqnarray*}
|\A_{n,r,s}(X)| & = & \sum_{i=s}^{2s-1}\left(\binom{i-1}{s-1}-\binom{i-t}{s-1}\right)\binom{n-i}{r-s}+\sum_{i=s}^{\min\{r+1,2s-1\}}\binom{i-t}{s-1}\binom{n-i-1}{r-s-1}\\
 &  & +\binom{r+2-t}{s-1}\binom{n-r-2}{r-s}+\sum_{i=r+3}^{2s-1}\binom{i-t-1}{s-2}\binom{n-i}{r-s}.
\end{eqnarray*}
Although in this case we have $t=1$, we address general $t$ so that
we may use this same argument in the cases that follow. Notice when
$s\ge3$, we have 
\[
\frac{\binom{n-i}{r-s}}{\binom{n-2}{r-2}}=\frac{(r-2)\cdots(r-s+1)}{(n-2)\cdots(n-s+1)}\cdot\frac{(n-i)\cdots(n-i-r+s+1)}{(n-s)\cdots(n-r+1)}\le\frac{(r-2)\cdots(r-s+1)}{(n-2)\cdots(n-s+1)}.
\]
Let $f(n,r,s)=\frac{(r-2)\cdots(r-s+1)}{(n-2)\cdots(n-s+1)}$ and
\begin{eqnarray*}
g(n,r,s) & = & \sum_{i=s}^{2s-1}\left(\binom{i-1}{s-1}-\binom{i-t}{s-1}\right)f(n,r,s)+\sum_{i=s}^{\min\{r+1,2s-1\}}\binom{i-1}{s-1}f(n,r,s-1)\\
 &  & +\binom{r+1}{s-1}f(n,r,s)+\sum_{i=r+3}^{2s-1}\binom{i-2}{s-2}f(n,r,s).
\end{eqnarray*}
This gives $\frac{|\A_{n,r,s}(X)|}{|\S(X)|}\le g(n,r,s)$. Notice
that this holds for all $2\le s\le r$, since the $\binom{r+2-t}{s-1}\binom{n-r-2}{r-s}+\sum_{i=r+3}^{2s-1}\binom{i-t-1}{s-2}\binom{n-i}{r-s}$
terms only appear for $2s-1\ge r+2$. Also, in the case $t=1$, the
first term is $0$ hence there are no such terms with $s=2$. Notice
that $f(n,r,s)$ is decreasing in $n$, so $g(n,r,s)$ is decreasing
in $n$. Thus to check the remaining cases, it suffices to check that
$|\A_{n,r,s}(X)|\le|\S(X)|$ for $3\le r\le241$, $2\le s\le r$,
and $n$ such that $g(n,r,s)>1$. This is not difficult with a computer,
and doing so completes the proof in the case $t=1$. %
\begin{comment}
Using a computer, we calculate $\frac{\partial}{\partial a}q(ar,r)$,
and maximize it to find that $\frac{\partial}{\partial a}q(ar,r)$
is negative, so $q(ar,r)$ is decreasing in $a$. Thus to show $|\A_{n,r,s}(X)|\le|\S(X)|$
for $n\ge\varphi^{2}r$, it suffices to evaluate it for $n=\lceil\varphi^{2}r\rceil$.
Also, we compute $\frac{\partial}{\partial r}q(ar,r)$, and minimize
it to find that it is always positive, so $q(ar,r)$ is increasing
in $r$. By Lemma \ref{l.bincoflim}, it increases to $\frac{1}{a}+\frac{2(a-1)}{a^{2}}$.
Thus it suffices to find for which $r$ we have $\frac{1}{a}+\frac{2(a-1)}{a^{2}}+3.25r^{3/2}(.954)^{r}<1$,
with $a=\varphi^{2}$. Using a computer, it is easy to check that
this holds for $r\ge241$. We then use the formulas in Lemma \ref{l.countform},
and a computer to check the remaining cases $3\le r\le240$, $2\le s\le r$,
we do indeed have $|\A_{n,r,s}(X)|\le|\S(X)|$ for $n=\lceil\varphi^{2}r\rceil$.
$ $
\end{comment}

\paragraph*{Case 2: $t=2$}

We first address the case $t=2,$ and $s=2$. By Lemma \ref{l.countform},
\[
|\A_{n,r,2}(X)|=2\binom{n-3}{r-3}-\binom{n-4}{r-4}+2\left(2\binom{n-4}{r-3}-\binom{n-5}{r-4}\right).
\]
By expanding the binomials, and pulling out common factors, we get
\[
|\S(X)|-|\A_{ar,r,2}(X)|=\frac{(n-5)!}{(r-2)!(n-r)!}(2n^{3}+(-7r-4)n^{2}+(7r^{2}+14r-4)n-2r^{3}-10r^{2}+4r).
\]
We will find for which $n$ we have $|\S(X)|-|\A_{n,r,2}(X)|\ge0$.
Notice we have the factorization
\[
2n^{3}+(-7r-4)n^{2}+(7r^{2}+14r-4)n-2r^{3}-10r^{2}+4r=(n-r)(2n^{2}+(-5r-4)n+2r^{2}+10r-4),
\]
so to find when $|\S(X)|-|\A_{n,r,2}(X)|\ge0$, we just use the quadratic
formula to get that $|\S(X)|-|\A_{n,r,2}(X)|\ge0$ for
\[
n\ge\frac{5r+4+\sqrt{9r^{2}-40r+48}}{4}.
\]
Notice that for $r\ge2$, $9r^{2}-40r+48\le9r^{2}$, so
\[
\frac{5r+4+\sqrt{9r^{2}-40r+48}}{4}\le\frac{8r+4}{4}=2r+1.
\]
This is less than the bound $n>\varphi^{2}r$.

For $s\ge3$, as discussed at the beginning of the proof, we have
\[
\frac{|\A_{n,r,s}(X)|}{|\S(X)|}\le\left(\frac{\binom{n-3}{r-3}+2\binom{n-4}{r-3}+3\binom{n-5}{r-3}+\binom{n-5}{r-4}+3\binom{n-6}{r-4}}{\binom{n-2}{r-2}}+3.25r^{3/2}(.954)^{r}\right)\frac{\binom{n-2}{r-2}}{|\S(X)|}.
\]
We use the same method as in Case 1. We expand the fraction in the
above expression, we get
\begin{eqnarray*}
q(n,r) & := & \frac{\binom{n-3}{r-3}+2\binom{n-4}{r-3}+3\binom{n-5}{r-3}+\binom{n-5}{r-4}+3\binom{n-6}{r-4}}{\binom{n-2}{r-2}}\\
 & = & \frac{(r-2)}{(n-2)}+\frac{2(r-2)(n-r)}{(n-2)(n-3)}+\frac{3(r-2)(n-r)(n-r-1)}{(n-2)(n-3)(n-4)}\\
 &  & \frac{(r-2)(r-3)(n-r)}{(n-2)(n-3)(n-4)}+\frac{3(r-2)(r-3)(n-r)(n-r-1)}{(n-2)(n-3)(n-4)(n-5)}.
\end{eqnarray*}
As in Case 1, we have
\[
q(n,r)\le\frac{1}{a}+\frac{2}{a}\left(1-\frac{1}{a}+\frac{3}{ar}\right)+\frac{3}{a}\left(1-\frac{1}{a}+\frac{3}{ar}\right)^{2}+\frac{1}{a^{2}}\left(1-\frac{1}{a}+\frac{4}{ar}\right)+\frac{3}{a^{2}}\left(1-\frac{1}{a}+\frac{4}{ar}\right)^{2}.
\]
This is clearly decreasing in $r$, so as above we may assume $r=200$.
Substituting $r=200$, it is easy to check that the expression is
decreasing in $a$, so it is sufficient to assume $a=\varphi^{2}$.
This gives $q(n,r)\le1.573$ for $r\ge200$. 

We now address $\frac{\binom{n-2}{r-2}}{|\S(X)|}$. As seen in the
proof of Lemma \ref{l.lastpart}, we have $\frac{\binom{n-2}{r-2}}{|\S(X)|}\le\frac{r-1}{(n-1)\left(1-\left(1-\frac{r-1}{n-1}\right)^{t}\right)}$,
and the right hand side of this inequality is decreasing in $a$,
so it suffices to assume $a=\varphi^{2}$. Putting all of this together,
it is sufficient to find for which $r$ we have 
\[
\left(1.573+3.25r^{3/2}(.954)^{r}\right)\frac{r-1}{(\varphi^{2}r-1)\left(1-\left(1-\frac{r-1}{\varphi^{2}r-1}\right)^{2}\right)}<1.
\]
The above inequality holds for $r\ge r_{0}=269$. Using the same method
with $g(n,r,s)$ as in Case $1$, we complete the proof in the case
$t=2$.%
\begin{comment}
Using a computer, we compute $\frac{\partial}{\partial a}q(ar,r)$,
and maximize it to find that it is always negative, so $q(ar,r)$
is decreasing in $a$. Thus to show $|\A_{n,r,s}(X)|\le|\S(X)|$ for
$n\ge\varphi^{2}r$, it suffices to evaluate for $n=\lceil\varphi^{2}r\rceil$.
We then compute $\frac{\partial}{\partial r}q(ar,r)$, and we find
it is increasing in $r$ for $a\ge2.5$. By Lemma \ref{l.bincoflim},
it increases to $\frac{2}{a^{2}}-\frac{1}{a^{3}}+3\frac{a-1}{a^{2}}+3\frac{(a-1)^{2}}{a^{4}}+3\frac{(a-1)^{2}}{a^{3}}$.
Similarly, we consider
\[
\frac{\binom{n-2}{r-2}}{|\S(X)|}=\frac{\binom{n-2}{r-2}}{\binom{n-1}{r-1}-\binom{n-3}{r-1}}.
\]
As above, we use a computer to check that $\frac{\binom{n-2}{r-2}}{|S(X)|}\le\frac{1}{a-a(1-\frac{1}{a})^{t}}$
for $t=2,3,4,5$ (this will be used later). Thus
\[
\frac{|\A_{n,r,s}(X)|}{|\S(X)|}\le\left(\frac{2}{a^{2}}-\frac{1}{a^{3}}+3\frac{a-1}{a^{2}}+3\frac{(a-1)^{2}}{a^{4}}+3\frac{(a-1)^{2}}{a^{3}}+3.25r^{3/2}(.954)^{r}\right)\left(\frac{1}{a-a(1-\frac{1}{a})^{2}}\right).
\]
We substitute $a=\varphi^{2}$, and use a computer to check that the
right hand side of the inequality above is less than 1 for $r\ge259$.
As before, we check all cases $3\le r\le258$, and $2\le s\le r$
to prove the result in the case $t=2$, and $s\ge3$.
\end{comment}

\paragraph*{Case 3: $t=3$}

The exceptional case $s=2$ was addressed in Section \ref{s.bigbad}.
For $s\ge3$, we have

\[
\frac{|\A_{n,r,s}(X)|}{|\S(X)|}\le\left(\frac{\binom{n-3}{r-3}+3\binom{n-4}{r-3}+5\binom{n-5}{r-3}+\binom{n-6}{r-4}}{\binom{n-2}{r-2}}+3.25r^{3/2}(.954)^{r}\right)\frac{\binom{n-2}{r-2}}{|\S(X)|}.
\]
We proceed with exactly the same method as in Cases $1$ and $2$.
The computations are very similar, so we omit them. %
\begin{comment}
As in Cases 1 and 2 we let 
\begin{eqnarray*}
q(n,r) & = & \frac{\binom{n-3}{r-3}+3\binom{n-4}{r-3}+5\binom{n-5}{r-3}+\binom{n-6}{r-4}}{\binom{n-2}{r-2}}\\
 & = & \frac{r-2}{n-2}+3\frac{(r-2)(n-r)}{(n-2)(n-3)}+5\frac{(r-2)(n-r)(n-r-1)}{(n-2)(n-3)(n-4)}\\
 &  & +\frac{(r-2)(r-3)(n-r)(n-r-1)}{(n-2)(n-3)(n-4)(n-5)}\\
 & \le & \frac{1}{a}+\frac{3}{a}(1-\frac{1}{a}+\frac{3}{ar})+\frac{5}{a}(1-\frac{1}{a}+\frac{3}{ar})+\frac{1}{a^{2}}(1-\frac{1}{a}+\frac{4}{ar})
\end{eqnarray*}

Decreasing in $r$ so set $r=200$. The resulting expression is decreasing
in $a$, so set $a=\varphi^{2}$. This gives $q(n,r)\le1.897$.

We find that 
\[
\left(1.897+3.25r^{3/2}(.954)^{r}\right)\frac{r-1}{(\varphi^{2}r-1)\left(1-\left(1-\frac{r-1}{\varphi^{2}r-1}\right)^{3}\right)}<1.
\]

for $r\ge r_{0}=249$. We then check the remaining values of $r$
as in case 1,2.
\end{comment}

\paragraph*{Case 4: $t\ge4$}

We first consider the case $t\ge4$, $s=2$. Notice that for $s=2$,
$|\A_{n,r,2}(\{2,\ldots,t_{1},r+1\})|=|\A_{n,r,2}(\{2,\ldots,t_{2},r+1\})|=\binom{n-2}{r-2}+2\binom{n-3}{r-2}$
for any $t_{1},t_{2}\ge4$. Since $|\S(X)|$ is increasing in $t$,
it suffices to check the case $t=4$. We will consider $|\S(X)|-|\A_{n,r,2}(X)|$.
We will pull out common factors of the binomial coefficients to obtain
a polynomial in $n$ and $r$. By analyzing this polynomial, we will
find a bound on $n$ which implies $|\S(X)|\ge|\A_{n,r,2}(X)|$. Although
we have already established that $|\S(X)|=\binom{n-1}{r-1}-\binom{n-5}{r-1}$,
it will prove simplest to use a different formula. We use the principle
of inclusion-exclusion to find that $|\S(X)|=4\binom{n-2}{r-2}-6\binom{n-3}{r-3}+4\binom{n-4}{r-4}-\binom{n-5}{r-5}$.
This gives
\[
|\S(X)|-|\A_{n,r,2}(X)|=\frac{(n-5)!}{(r-2)!(n-r)!}(n^{3}+(-4r-1)n^{2}+(4r^{2}+8r-6)n-r^{3}-7r^{2}+6r),
\]
which factors as
\[
\frac{(n-5)!}{(r-2)!(n-r)!}(n-r)\left(n-\frac{3r+1+\sqrt{5r^{2}-22r+25}}{2}\right)\left(n-\frac{3r+1-\sqrt{5r^{2}-22r+25}}{2}\right).
\]
This is the same factorization as seen in Section \ref{s.bigbad},
so we once again obtain the desired bound $n\ge\varphi^{2}r$ implies
$|\A_{n,r,2}(X)|\le|\S(X)|$. 

We must now address the case $t\ge4$ and $s\ge3$. Similar to the
case $s=2$, we have $D(n,r,3,t_{1})=D(n,r,3,t_{2})=\binom{n-3}{r-3}+3\binom{n-4}{r-3}+6\binom{n-5}{r-3}$
for any $t_{1},t_{2}\ge4$, so it suffices to check the case $t=4$.
Thus
\[
\frac{|\A_{n,r,s}(X)|}{|\S(X)|}\le\left(\frac{\binom{n-3}{r-3}+3\binom{n-4}{r-3}+6\binom{n-5}{r-3}}{\binom{n-2}{r-2}}+3.25r^{3/2}(.954)^{r}\right)\left(\frac{r-1}{(n-1)\left(1-(1-\frac{r-1}{n-1})^{4}\right)}\right).
\]
We use the same method demonstrated in Cases 1 and 2 to complete the
proof of Theorem \ref{t.onegen}. The details are similar to the computations
above, so they have been omitted. \end{proof} %
\section{Proof of Theorem \ref{t.multgen} \label{s.multgen}}

In this section, we prove Theorem \ref{t.multgen}. Before the proof,
we need the following lemma.

\begin{prop}\label{l.nicegens}

Let 
\[
\B=F(r,n,\{\{1,r+1\}\})\cup F(r,n,\{\{2,3,r+2\}\})\cup\bigcup_{s=3}^{r}\A_{n,r,s}.
\]
Let $\A$ be an intersecting family. If $\A\not\subset\S$ and $\A\not\subset\A_{n,r,2}$,
then $\A\subset\B$.

\end{prop}

\begin{proof}

Consider an element $A=\{a_{1},\ldots,a_{r}\}\subset\A$. We begin
by assuming $a_{1}=1$, and we consider the maximal possible value
for $a_{2}$. For the sake of contradiction, assume that $a_{2}\ge r+2$.
Since $a_{i}>a_{i-1}$, we have $a_{i}\ge r+i$ for $i\ge2$. Since
$\A\not\subset\S$, there exists some $B=\{b_{1},\ldots,b_{r}\}\subset\A$
such that $b_{1}\ge2$. By Proposition \ref{p.intcond}, there exists
a pair $i,j$ such that $i+j>\max\{a_{i},b_{j}\}$. If $i=1$, then
$1+j>b_{j}$, which implies $b_{j}=j$, since for any $B\subset\binom{[n]}{r}$,
we have $b_{j}\ge j$. This implies $b_{1}=1$, which is a contradiction.
So $i\ge2$ and we have $i+j>a_{i}\ge r+i$, hence $j>r$, which is
impossible. Thus we cannot have $a_{2}\ge r+2$, so all $A\subset\A$
with smallest element $1$ are contained in $F(n,r,\{\{1,r+1\}\})$.

Now assume that $A\prec[s,2s-1]$ for some $s\ge2$. If $s\ge3$,
then $A\subset\A_{n,r,s}\subset\B$. Thus we may assume $s=2$. If
$a_{1}=1$, we know $A\subset\B$ by the previous paragraph, so we
may assume $a_{1}=2$, which implies $a_{2}=3$. To show $A\subset F(r,n,\{\{2,3,r+2\}\}$,
we argue as in the previous paragraph. If $a_{3}\ge r+3$, then $a_{i}\ge r+i$
for $i\ge3$. Since $\A\not\subset\A_{n,r,2}$, there exists some
$B\subset\A$ with $B\not\prec[2,3]$. This implies $b_{2}\ge4$.
By Proposition \ref{p.intcond}, there exists a pair $i,j$ with $i+j>\max\{a_{i},b_{j}\}$.
As in the previous paragraph, if $i\ge3$, then $j>r$ which is impossible,
so $i\in\{1,2\}$. 

We first show we cannot have $i=1$. If $i=1$, then we cannot have
$j=1$ since $a_{2}=2$. If $j\ge2$, then notice that since $i=1$,
we have $1+j>b_{j}$, but we always have $b_{j}\ge j$, so $b_{j}=j$.
Since $b_{j}<b_{j+1}$, this implies $b_{2}=2$, which contradicts
$b_{2}\ge4$. Thus $i\ne1$. Now for a contradiction assume $i=2$.
If $j=1$, then we have $3>a_{2}=3$ which is impossible, hence $j\ge2$.
However, if $j\ge2$, then we have $2+j>b_{j}\ge j$, so $b_{j}\in\{j,j+1\}$.
Since $b_{j-1}\le b_{j}-1$, this implies $b_{2}\in\{2,3\}$, which
is false. Thus $i\not\in\{1,2\}$. This is impossible if $a_{3}\ge r+3$,
so we have $a_{3}\le r+2$, so $A\subset F(r,n,\{2,3,r+2\})$.\end{proof}

We restate Theorem \ref{t.multgen} for easy reference.

\setcounter{count}{3}

\begin{thm} Let $r\ge3$ and let $X$ be eventually EKR with $|X|\le r$.
For each $\varepsilon>0$, there exists an $r_{0}$ such that for
$r>r_{0}$, the condition $n>(2+\varepsilon)r^{2}$ implies $X$ is
EKR. Furthermore, 

\begin{enumerate}
\item If $|X|=1$ and $r\ge 12$, then $n>2r^2$ implies $X$ is EKR.
\item If $|X|=2$ and $r\ge 14$, then $n>2r^2$ implies $X$ is EKR.
\item If $|X|=3$ and $r\ge 14$, then $n>3r^2$ implies $X$ is EKR.
\item If $|X|=t\ge 4$ and $r\ge \max\{11,t\}$, then $n>tr^2$ implies $X$ is EKR.
\end{enumerate}
\end{thm}

\setcounter{count}{20}
\begin{proof}

The proof breaks into several cases, depending on $|X|=t$. In each
case, we use the bound $|\A(X)|\le|\B(X)|$, which follows immediately
from Lemma \ref{l.nicegens}. We will count $|\B(X)|$ by considering
each of $F(r,n,\{\{1,r+1\}\})$, $F(r,n,\{\{2,3,r+2\}\})$, and $\bigcup_{s=3}^{r}\A_{n,r,s}$
separately, and then we consider the ratio $\frac{|\B(X)|}{|\S(X)|}$.
Using the preliminaries derived in Section \ref{s.onegenprelims},
we find how large $n$ must be for this ratio to be less than 1. As
in the proof of Theorem \ref{t.onegen}, we will define $a=n/r$.

\paragraph{Case 1: $t=1$}

We begin by considering $\frac{|\C(X)|}{\binom{n-2}{r-2}}$ for various
subfamilies $\C\subset\B.$ We will then combine the bounds to bound
$|\B(X)|$.

Consider the family $\C_{1}=F(r,n,\{1,r+1\})$. By summing over possible
values for the second element $i$ of a set, $2\le i\le r+2$, and
choosing the $r-3$ possible remaining elements, we see 
\[
|\C_{1}(X)|=\sum_{i=2}^{r+1}\binom{n-i-1}{r-3}.
\]
Notice that $\frac{r-2}{n-2}\le\frac{r}{n}=\frac{1}{a}$, and $\frac{n-i-1-j}{n-3-j}\le1$
for $0\le j\le r-4$, so
\[
\frac{\binom{n-i-1}{r-3}}{\binom{n-2}{r-2}}=\frac{r-2}{n-2}\cdot\frac{(n-i-1)(n-i-2)\cdots(n-i-r+3)}{(n-3)(n-4)\cdots(n-r+1)}\le\frac{1}{a}.
\]
This gives 
\[
\frac{|\C_{1}(X)|}{\binom{n-2}{r-2}}\le\frac{r}{a}.
\]

We now consider $\C_{2}=F(r,n,\{\{2,3,r+2\}\})$, and we count the
number of elements of $\C_{2}(X)$ that are not in $\C_{1}(X)$. All
elements of $\C_{2}(X)$ that have smallest element $1$ have already
been counted in $\C_{1}(X)$, so we just need to count the number
of elements of $\C_{2}(X)$ with first two elements $2,3$. We sum
over the possible values $i$ for the third element, $3\le i\le r+1$,
and then we also account for elements of $\C_{2}(X)$ with first three
elements $2,3,r+2$. The number of these elements is
\[
\sum_{i=3}^{r+1}\binom{n-i-1}{r-4}+\binom{n-r-2}{r-3}.
\]
Using the same method used above, we have ${\displaystyle \frac{\binom{n-i-1}{r-4}}{\binom{n-2}{r-2}}\le\frac{1}{a^{2}}}$
and ${\displaystyle \frac{\binom{n-r-2}{r-3}}{\binom{n-2}{r-2}}\le\frac{1}{a}}$,
hence
\[
\frac{\sum_{i=3}^{r+1}\binom{n-i-1}{r-4}+\binom{n-r-2}{r-3}}{\binom{n-2}{r-2}}\le\frac{r-1}{a^{2}}+\frac{1}{a}.
\]
Combining this with the bound on $|\C_{1}(X)|$, we have
\[
\frac{|(\C_{1}\cup\C_{2})(X)|}{\binom{n-2}{r-2}}\le\frac{r}{a}+\frac{r-1}{a^{2}}+\frac{1}{a}.
\]

We now consider $\C_{3}=\bigcup_{s=3}^{r}\A_{n,r,s}$. We consider
each $\A_{n,r,s}(X)$ separately, starting with the part of of $|\A_{n,r,s}(X)|$
that goes to zero exponentially in $r$. Let
\[
E(n,r)=(r-3)\cdot\frac{(r-1)a(ar-1)e^{4}}{2(r+3)^{2}(2\pi)^{2}}\cdot\sqrt{\frac{(r+3)(ar-r-3)(ar-r)}{2\pi(r-1)(ar-r-4)a}}\cdot\left(\frac{2\sqrt{2}(a-1)^{2a-2}}{(a-\frac{3}{2})^{a-\frac{3}{2}}a^{a}}\right)^{r}.
\]
Notice that this is $(r-3)T(ar,r,r,r/2)$ from \eqref{eq:exppart}
in the proof of Lemma \ref{l.goestozero}. Recall that $s=r/2$ and
$i=r$ maximizes the term. The factor of $r-3$ comes from the fact
that there are at most $r-3$ terms in the sum $\sum_{i=r+3}^{2s-1}\binom{i-t-1}{s-2}\binom{n-i}{r-s}$
from Lemma \ref{l.goestozero}. We now consider the remaining part
of $\A_{n,r,s}(X)$, namely ${\displaystyle D(n,r,s,t)=\sum_{i=s}^{2s-1}\left(\binom{i-1}{s-1}-\binom{i-t}{s-1}\right)\binom{n-i}{r-s}+\sum_{i=s}^{\min\{r+1,2s-1\}}\binom{i-t}{s-1}\binom{n-i-1}{r-s-1}}$.
From Lemmas \ref{l.goestozero}, \ref{l.ynofactor}, we know for $s\ge4$,
\[
\frac{|\A_{n,r,s}(X)|}{\binom{n-2}{r-2}}\le\frac{D(n,r,4,1)}{\binom{n-2}{r-2}}+E(n,r).
\]
Thus 
\[
\frac{|\C_{3}(X)|}{\binom{n-2}{r-2}}\le\sum_{s=3}^{r}\frac{|\A_{n,r,s}(X)|}{\binom{n-2}{r-2}}\le(r-2)E(n,r)+\frac{D(n,r,3,1)+(r-3)D(n,r,4,1)}{\binom{n-2}{r-2}}.
\]
Notice that by considering $D(n,r,3,1)+(r-3)D(n,r,4,1)$ instead of
$(r-2)D(n,r,3,1)$, we obtain a better approximation. We could have
considered higher terms, but this does not improve the asymptotic
bounds, and leads to a more complicated expression. Combining the
bounds obtained by considering $\C_{1}(X)$, $\C_{2}(X)$, and $\C_{3}(X)$,
and substituting for $D(n,r,s,t)$, we have 
\[
\frac{|\B(X)|}{\binom{n-2}{r-2}}\le|\C_{1}(X)\cup\C_{2}(X)|+|\C_{3}(X)|\le\frac{r}{a}+\frac{r-1}{a^{2}}+\frac{1}{a}+(r-2)E(n,r)+D(n,r,3,1)+(r-3)D(n,r,4,1).
\]
Expanding $D(n,r,3,1)$ and $D(n,r,4,1)$, we obtain
\begin{eqnarray}
\frac{|\B(X)|}{\binom{n-2}{r-2}} & \le & \frac{r}{a}+\frac{r-1}{a^{2}}+\frac{1}{a}+(r-2)E(n,r)\label{eq:BXt1}\\
 &  & +\frac{\binom{n-4}{r-4}+3\binom{n-5}{r-4}+6\binom{n-6}{r-4}+(r-3)\left(\binom{n-5}{r-5}+4\binom{n-6}{r-5}+10\binom{n-7}{r-5}+20\binom{n-8}{r-5}\right)}{\binom{n-2}{r-2}}.\nonumber 
\end{eqnarray}
As $r$ gets large, the right hand side of \eqref{eq:BXt1} goes to
\[
\frac{r}{a}+\frac{r-1}{a^{2}}+\frac{1}{a}+\frac{1}{a^{2}}+3\frac{(a-1)}{a^{3}}+6\frac{(a-1)^{2}}{a^{4}}+(r-3)\left(\frac{1}{a^{3}}+4\frac{(a-1)}{a^{4}}+10\frac{(a-1)^{2}}{a^{5}}+20\frac{(a-1)^{3}}{a^{5}}\right).
\]

If we take $a=cr$ for some constant $c$, by ignoring terms that
go to $0$ as $r$ grows, this simplifies to $1/c$, which comes from
the $r/a$ term. Thus for any $\varepsilon>0$, for sufficiently large
$r$, we have $X$ is EKR for $n>(1+\varepsilon)r^{2}$. 

For example, if $\varepsilon=.1$, then by substituting $a=1.1r$
into the right hand side of \eqref{eq:BXt1} and using a computer
to check when this is less than $1$, we have that $X$ is EKR for
$n>1.1r^{2}$ for $r\ge28$. If we increase $\varepsilon$ to $\varepsilon=1$,
then we get $n>2r^{2}$ implies $X$ is EKR for $r\ge12$. Note that
the coefficient $(1+\varepsilon)$ differs from the coefficient $(2+\varepsilon)$
in Theorem \ref{t.multgen}. We will see that the cases $t\ge3$ determine
the coefficient $(2+\varepsilon)$. Also, note that the choices $\varepsilon=.1$
and $\varepsilon=1$ were arbitrary. These examples were computed
to give a bound on $n$ that is not dependent on $r$ being arbitrarily
large.

We remark that in the case $t=1$, it would have been sufficient to
simply consider $(r-2)D(n,r,3,1)$ instead of $D(n,r,3,1)+(r-3)D(n,r,4,1)$.
This is because the largest binomial coefficient in $D(n,r,3,1)$
is $\binom{n-4}{r-4}$, and $(r-2)\binom{n-4}{r-4}/\binom{n-2}{r-2}$
is approximately $r/a^{2}$ for large $r$. However, in this case
we establish a pattern that will repeated in other cases, in which
it will be necessary to consider $D(n,r,3,1)+(r-3)D(n,r,4,1)$.

\paragraph{Case 2: $t=2$}

We repeat a similar computation to the case $t=1$. Consider the family
$\C_{1}=F(r,n,\{1,r+1\})$. We begin by counting the the number of
elements of $\C_{1}$ containing $r+2$, which is $\sum_{i=2}^{r+1}\binom{n-i-1}{r-3}$
as in Case $1$. We also count the number of elements containing $4$.
If an element of $\C_{1}$ contains 4, the first elements must be
$\{1,2,3,4\}$, $\{1,2,4\}$, $\{1,3,4\}$, or $\{1,4\}$. There are,
respectively, $\binom{n-4}{r-4}$, $\binom{n-4}{r-3}$, $\binom{n-4}{r-3}$,
$\binom{n-4}{r-2}$ elements of $\C$ of each form. We have double
counted some elements, but a simpler bound suffices for our purposes.
This gives
\[
|\C_{1}(X)|\le\sum_{i=2}^{r+1}\binom{n-i-1}{r-3}+\binom{n-4}{r-4}+2\binom{n-4}{r-3}+\binom{n-4}{r-2},
\]
hence
\[
\frac{|\C_{1}(X)|}{\binom{n-2}{r-2}}\le\frac{r}{a}+\frac{\binom{n-4}{r-4}+2\binom{n-4}{r-3}+\binom{n-4}{r-2}}{\binom{n-2}{r-2}}.
\]

We now consider $\C_{2}=F(r,n,\{\{2,3,r+2\}\})$, and we count the
number of elements of $\C_{2}(X)$ that are not in $\C_{1}(X)$. As
before, we only need to count the elements with first two elements
$\{2,3\}$. There are $\binom{n-4}{r-3}$ elements with first three
elements $\{2,3,4\}$, and $\sum_{i=5}^{r+1}\binom{n-i-1}{r-4}+\binom{n-r-2}{r-3}$
that contain $r+2$ but not $4$, thus $\C_{2}(X)$ contributes
\[
\sum_{i=5}^{r+2}\binom{n-i-1}{r-4}+\binom{n-4}{r-3}+\binom{n-r-2}{r-3}
\]
elements not already counted in $\C_{1}(X)$. 

We now consider $\C_{3}=\bigcup_{s=3}^{r}\A_{n,r,s}$. We have
\[
\frac{|\C_{3}(X)|}{\binom{n-2}{r-2}}\le(r-2)E(n,r)+\frac{D(n,r,3,2)+(r-3)D(n,r,4,2)}{\binom{n-2}{r-2}}.
\]
Notice $\frac{|\B(X)|}{|\S(X)|}=\frac{|\B(X)|}{|\binom{n-2}{r-2}}\frac{\binom{n-2}{r-2}}{|\S(X)|}$.
As seen in Lemma \ref{l.lastpart}, we have $\frac{\binom{n-2}{r-2}}{|\S(X)|}\le\frac{\frac{r-1}{n-1}}{1-(1-\frac{r-1}{n-1})^{2}}$.
Combining the bounds obtained by considering $\C_{1}(X)$, $\C_{2}(X)$,
and $\C_{3}(X)$, we have
\begin{eqnarray}
\frac{|\B(X)|}{|\S(X)|} & \le & \frac{\frac{r-1}{n-1}}{1-\left(1-\frac{r-1}{n-1}\right)^{2}}\left(\frac{r}{a}+\frac{\binom{n-4}{r-4}+2\binom{n-4}{r-3}+\binom{n-4}{r-2}}{\binom{n-2}{r-2}}+\frac{r-1}{a^{2}}+\frac{\binom{n-4}{r-3}+\binom{n-r-2}{r-3}}{\binom{n-2}{r-2}}\right.\nonumber \\
 &  & +(r-2)E(n,r)+\frac{\binom{n-3}{r-3}+2\binom{n-4}{r-3}+3\binom{n-5}{r-3}+\binom{n-5}{r-4}+3\binom{n-6}{r-4}}{\binom{n-2}{r-2}}\label{eq:BXt2}\\
 &  & \left.(r-3)\left(\frac{\binom{n-4}{r-4}+3\binom{n-5}{r-4}+6\binom{n-6}{r-4}+10\binom{n-7}{r-4}+\binom{n-6}{r-5}+4\binom{n-7}{r-5}+10\binom{n-8}{r-5}}{\binom{n-2}{r-2}}\right)\right).\nonumber 
\end{eqnarray}
\begin{comment}
As $r$ grows large, this is less than
\begin{eqnarray*}
 &  & \frac{\frac{1}{a}}{1-(1-\frac{1}{a})^{2}}\left(\frac{r}{a}+\frac{1}{a^{2}}+2\frac{a-1}{a^{2}}+\frac{(a-1)^{2}}{a^{2}}+\frac{r-1}{a^{2}}+\frac{a-1}{a^{2}}+\frac{(a-1)^{r-1}}{a^{r}}\right.\\
 &  & +\frac{1}{a^{2}}+2\frac{a-1}{a^{2}}+3\frac{(a-1)^{2}}{a^{3}}+\frac{a-1}{a^{3}}+3\frac{(a-1)^{2}}{a^{4}}\\
 &  & \left.+(r-3)\left(\frac{1}{a^{2}}+3\frac{a-1}{a^{3}}+6\frac{(a-1)^{2}}{a^{4}}+10\frac{(a-1)^{3}}{a^{5}}+\frac{a-1}{a^{4}}+4\frac{(a-1)^{3}}{a^{5}}+10\frac{(a-1)^{3}}{a^{6}}\right)\right).
\end{eqnarray*}
Substituting $a=cr$, ignoring terms that go to $0$, and using Lemma
\ref{l.lastpart} on the first factor, this goes to
\end{comment}
Using the same method demonstrated in Case $1$ of considering large
$r$ and approximating, then substituting $a=cr$, we may bound this
for large $r$ by $\frac{1}{2}\left(\frac{1}{c}+1\right).$ Thus for
any $\varepsilon>0$, for sufficiently large $r$, we have $X$ is
EKR for $n>(1+\varepsilon)r^{2}$. For example, using \eqref{eq:BXt2},
we have for $\varepsilon=.1$, we have $r\ge268$, $n>1.1r^{2}$ implies
$X$ is EKR. For $\varepsilon=1$, we have $n>2r^{2}$ implies $X$
is EKR for $r\ge14$.

\paragraph{Case 3: $t=3$}

We repeat a similar computation to the case $t=1$. Consider the family
$\C_{1}=F(r,n,\{1,r+1\})$. We begin by counting the the number of
elements of $\C_{1}$ containing $r+2$, which is $\sum_{i=2}^{r+1}\binom{n-i-1}{r-3}$
as in Case $1$. We also count the number of elements containing $2,4$.
If an element of $\C_{1}$ contains $2$ or 4, the first elements
must be $\{1,2\}$, $\{1,3,4\}$, or $\{1,4\}$. There are, respectively,
$\binom{n-2}{r-2}$, $\binom{n-4}{r-3}$, and $\binom{n-4}{r-2}$
elements of $\C$ of each form. We have double counted some elements,
but this is sufficient for our purposes. This gives
\[
|\C_{1}(X)|\le\sum_{i=2}^{r+1}\binom{n-i-1}{r-3}+\binom{n-2}{r-2}+\binom{n-4}{r-3}+\binom{n-4}{r-2},
\]
hence 
\[
\frac{|\C_{1}(X)|}{\binom{n-2}{r-2}}\le\frac{r}{a}+\frac{\binom{n-2}{r-2}+\binom{n-4}{r-3}+\binom{n-4}{r-2}}{\binom{n-2}{r-2}}.
\]

We now consider $\C_{2}=F(r,n,\{\{2,3,r+2\}\})$, and we count the
number of elements of $\C_{2}(X)$ that are not in $\C_{1}(X)$. As
before, we only need to count the elements with first two elements
$\{2,3\}$. Everything with first element $2$ is contained in $\C_{2}(X)$,
which adds an extra
\[
\sum_{i=4}^{r+2}\binom{n-i}{r-3}
\]
elements. Notice that $\binom{n-i}{r-3}\le1/a$ as seen in Case $1$,
hence $\frac{\sum_{i=4}^{r+2}\binom{n-i}{r-3}}{\binom{n-2}{r-2}}\le\frac{r-1}{a}$. 

We now consider $\C_{3}=\bigcup_{s=3}^{r}\A_{n,r,s}$. As above, we
have
\[
\frac{|\C_{3}(X)|}{\binom{n-2}{r-2}}\le(r-2)E(n,r)+\frac{D(n,r,3,3)+(r-3)D(n,r,4,3)}{\binom{n-2}{r-2}}.
\]
Combining the bounds obtained by considering $\C_{1}(X)$, $\C_{2}(X)$,
and $\C_{3}(X)$, and substituting for $D(n,r,s,t)$, we have
\begin{eqnarray}
\frac{|\B(X)|}{|\S(X)|} & \le & \frac{\frac{r-1}{n-1}}{1-\left(1-\frac{r-1}{n-1}\right)^{3}}\left(\frac{r}{a}+\frac{\binom{n-2}{r-2}+\binom{n-4}{r-3}+\binom{n-4}{r-2}}{\binom{n-2}{r-2}}+\frac{r-1}{a}\right.\nonumber \\
 &  & +(r-2)E(n,r)+\frac{\binom{n-3}{r-3}+3\binom{n-4}{r-3}+5\binom{n-5}{r-3}+\binom{n-6}{r-4}}{\binom{n-2}{r-2}}\label{eq:BXt3}\\
 &  & \left.+(r-3)\left(\frac{\binom{n-4}{r-4}+4\binom{n-5}{r-4}+9\binom{n-6}{r-4}+16\binom{n-7}{r-4}+\binom{n-7}{r-5}+4\binom{n-8}{r-5}}{\binom{n-2}{r-2}}\right)\right).\nonumber 
\end{eqnarray}
\begin{comment}
As $r$ gets large, this is
\begin{eqnarray*}
 &  & \frac{\frac{1}{a}}{1-(1-\frac{1}{a})^{3}}\left(\frac{r}{a}+1+\frac{a-1}{a^{2}}+\frac{(a-1)^{2}}{a^{2}}+\frac{r-1}{a}\right.\\
 &  & +\frac{1}{a}+3\frac{a-1}{a^{2}}+5\frac{(a-1)^{2}}{a^{3}}+\frac{(a-1)^{2}}{a^{4}}\\
 &  & +\left.(r-3)\left(\frac{1}{a^{2}}+4\frac{a-1}{a^{3}}+9\frac{(a-1)^{2}}{a^{4}}+16\frac{(a-1)^{3}}{a^{5}}+\frac{(a-1)^{2}}{a^{5}}+4\frac{(a-1)^{3}}{a^{6}}\right)\right).
\end{eqnarray*}
Substituting $a=cr$, ignoring terms that go to $0$, and using Lemma
\ref{l.lastpart} on the first term, this goes to
\end{comment}
Using the same method demonstrated in Case $1$ of considering large
$r$ and approximating, then substituting $a=cr$, we may bound this
for large $r$ by $\frac{1}{3}\left(\frac{2}{c}+2\right)$. Thus for
any $\varepsilon>0$, for sufficiently large $r$, we have $X$ is
EKR for $n>(2+\varepsilon)r^{2}$. For example, using \eqref{eq:BXt3},
we have for $\varepsilon=.1$, we have $r\ge237$, $n>2.1r^{2}$ implies
$X$ is EKR. For $\varepsilon=1$, we have $n>3r^{2}$ implies $X$
is EKR for $r\ge14$.

\paragraph{Case 4: $t\ge4$}

Consider the family $\C_{1}=F(r,n,\{1,r+1\})$. We first count the
number of elements of $\C_{1}$ containing an element of $\{2,3,\ldots,t\}$.
All elements of $\C_{1}$ containing $2$ have first two elements
$\{1,2\}$, so there are $\binom{n-2}{r-2}$ such elements. We must
now account for sets containing $3$, but not $2$. Notice that all
such sets have first two elements $\{1,3\}$. There are $\binom{n-3}{r-2}$
of these. In general, for an set $A\subset\C_{1}$ to contain $i$,
but none of $2,\ldots,i-1$, the first two elements must be $\{1,i\}$.
There are $\binom{n-i}{r-2}$ such sets. This gives a total of
\[
\sum_{i=2}^{t}\binom{n-i}{r-2}
\]
elements of $\C_{1}$ containing one of $2,\ldots,t$. We must then
count the number of sets containing none of $2,\ldots,t$, but containing
$r+2$. As before, we consider sets with second element $i$, for
$t+1\le i\le r+1$, which gives $\sum_{i=t+1}^{r+1}\binom{n-i-1}{r-3}$.
Thus
\[
|\C_{1}(X)|=\sum_{i=2}^{t}\binom{n-i}{r-2}+\sum_{i=t+1}^{r+1}\binom{n-i-1}{r-3}.
\]
Notice that $\binom{n-i}{r-2}/\binom{n-2}{r-2}\le1$, and $\binom{n-i-1}{r-3}\le1/a$,
so
\[
\frac{|\C_{1}(X)|}{\binom{n-2}{r-2}}\le t-1+\frac{r-t+1}{a}.
\]

We now consider $\C_{2}=F(r,n,\{\{2,3,r+2\}\})$, and we count the
number of elements of $\C_{2}(X)$ that are not in $\C_{1}(X)$. As
before, we only need to count the elements with first two elements
$\{2,3\}$. Everything with first element $2$ is contained in $\C_{2}(X)$,
which adds an extra
\[
\sum_{i=4}^{r+2}\binom{n-i}{r-3}
\]
elements. Notice that $\binom{n-i}{r-3}\le1/a$ as seen in Case $1$,
hence $\frac{\sum_{i=4}^{r+2}\binom{n-i}{r-3}}{\binom{n-2}{r-2}}\le\frac{r-1}{a}$. 

We now consider $\C_{3}=\bigcup_{s=3}^{r}\A_{n,r,s}$. We first find
which $t$ will maximize $D(n,r,3,t)$ and $D(n,r,4,t)$. Notice for
$r$ such that $\min\{r+1,2s-1\}=2s-1$, we have
\[
D(n,r,s,t)=\sum_{i=s}^{2s-1}\left(\binom{i-1}{s-1}-\binom{i-t}{s-1}\right)\binom{n-i}{r-s}+\binom{i-t}{s-1}\binom{n-i-1}{r-s-1}.
\]
Since $\binom{n-i-1}{r-s-1}-\binom{n-i}{r-s}=-\binom{n-i-1}{r-s}$,
we have
\[
D(n,r,s,t)=\sum_{i=s}^{2s-1}\binom{i-1}{s-1}\binom{n-i}{r-s}-\binom{i-t}{s-1}\binom{n-i-1}{r-s}.
\]
This is increasing in $t$. We are only concerned with the case $s\in\{3,4\}$,
and in this case, $D(n,r,s,t)$ will stop increasing in $t$ once
$\binom{i-t}{s-1}=0$ for all $i$, and this happens when $2s-1-t<s-1$,
which gives $t>s\ge4$. Thus we may assume $t=5$. In this case we
have 
\[
\frac{|\C_{3}(X)|}{\binom{n-2}{r-2}}\le(r-2)E(n,r)+\frac{D(n,r,3,5)+(r-3)D(n,r,4,5)}{\binom{n-2}{r-2}}.
\]
Thus we have
\begin{eqnarray}
\frac{|\B(X)|}{|\S(X)|} & \le & \frac{\frac{r-1}{n-1}}{1-\left(1-\frac{r-1}{n-1}\right)^{t}}\left(t-1+\frac{r-t+1}{a}+\frac{r-1}{a}+(r-2)E(n,r)\right.\label{eq:BXt4}\\
 &  & \left.\frac{\binom{n-3}{r-3}+3\binom{n-4}{r-3}+6\binom{n-5}{r-3}+(r-3)\left(\binom{n-4}{r-4}+4\binom{n-5}{r-4}+10\binom{n-6}{r-4}+20\binom{n-7}{r-4}\right)}{\binom{n-2}{r-2}}\right).\nonumber 
\end{eqnarray}
Letting $r$ be large and substituting $a=cr$, this is approximately
\[
\frac{1}{t}\left(t-1+\frac{2}{c}-\frac{t}{cr}\right).
\]
Notice that $t$ may be as large as $r$, so we cannot ignore $t/cr$.
Thus for any $\varepsilon>0$, there exists an $r_{0}$ such that
for $r\ge r_{0}$, we have $X$ is EKR for $n>(2+\varepsilon)r^{2}$. 

We now compute some examples of how large $r_{0}$ must be. Notice
that $\binom{n-i}{r-3}\le\binom{n-3}{r-3}\le1/a$ for $i\ge3$ and
$\binom{n-j}{r-4}\le\binom{n-4}{r-4}\le1/a^{2}$ for $j\ge4$. Thus
the right hand side of \eqref{eq:BXt4} is less than
\[
BX(n,r,t)=\frac{\frac{r-1}{n-1}}{1-\left(1-\frac{r-1}{n-1}\right)^{t}}\left(t-1+\frac{r-t+1}{a}+\frac{r-1}{a}+(r-2)E(n,r)+\frac{10}{a}+\frac{35(r-3)}{a^{2}}\right).
\]
Notice $BX(n,r,t)$ is decreasing in $n$, since the first factor
is decreasing in $n$ by Lemma \ref{l.lastpart}, and $E(n,r)$ is
decreasing in $n$. We consider the case $n=tr^{2}$. Using a computer,
it is easy to check that $BX(tr^{2},r,t)$ is increasing in $t$.
Since $t\le r$, it is sufficient to check the case $t=r$. We then
check that $B(r^{3},r,r)$ is increasing in $r$ for $r\ge11$. If
may be checked by tedious repetitions of L'Hopital's rule that $\lim_{r\rightarrow\infty}B(r^{3},r,r)=1$,
so $B(r^{3},r,r)\le1$. Thus for $r\ge t\ge11$, we have $n\ge tr^{2}$
implies $X$ is EKR. Checking the remaining cases, we see 

For $t=4$ and $r\ge11$, we have $n\ge4r^{2}$ implies $X$ is EKR. 

For $t=5$ and $r\ge10$, we have $n\ge5r^{2}$ implies $X$ is EKR.

For $t\in\{6,7,8,9\}$ and $r\ge9$, we have $n\ge tr^{2}$ implies
$X$ is EKR.

For $t=5$ and $r\ge10$, we have $n\ge10r^{2}$ implies $X$ is EKR.\end{proof}
\begin{comment}
We now want some examples, as in previous cases. Let $B(n,r)=\frac{\binom{n-3}{r-3}+3\binom{n-4}{r-3}+6\binom{n-5}{r-3}+(r-3)\left(\binom{n-4}{r-4}+4\binom{n-5}{r-4}+10\binom{n-6}{r-4}+20\binom{n-7}{r-4}\right)}{\binom{n-2}{r-2}}$.
Using a computer, and the same methods used in the proof of Theorem
\ref{t.onegen}, it is not hard to show that if we substitute $n=cr^{2}$,
then for $c\ge4$ and $r\ge5$, we have $B(cr^{2},r)$ is decreasing
in $r$ for fixed $c$, and is decreasing in $c$ for fixed $r$.
Similarly, for $c\ge4$, $r\ge5$, we have $(r-2)E(cr^{2},r)$ is
decreasing in $c$ for fixed $r$ and decreasing in $r$ for fixed
$c$. Picking some values, $r=15$, $c=4$, and substituting, we have
\[
\frac{|\B(X)|}{|\S(X)|}\le\frac{\frac{r-1}{n-1}}{1-\left(1-\frac{r-1}{n-1}\right)^{t}}\left(t-1+\frac{r-t+1}{a}+\frac{r-1}{a}+.221\right)
\]

Say that $a=tr$. 
\end{comment}
\section{A generating function for $|\A(X)|$\label{s.genfun}}

In this section, we introduce a generating function that can be used
to compute $|\A(X)|$ for any compressed family $\A$ and $X\subset[n]$.
This method of calculating $|\A(X)|$ is much faster in practice than
enumerating elements of $\A$ and checking if each intersects $X$.
In our experiments, the method given below seemed to be $40$ times
as fast.

Before beginning we note that for any set of generators $\emptyset\ne G\subset2^{[n]}$,
we may obtain a set of generators $\G'\subset\binom{[n]}{r}$ such
that $F(r,n,\G)=F(r,n,\G')$ and $|\G'|\le|\G|$. Indeed, if $G\in\G$,
and $|G|>r$, there is no $B\in\binom{[n]}{r}$ such that $B\prec G$.
If $|G|\le r$, then $G'=G\cup[n-(r-|G|)+1,n]\in\binom{[n]}{r}$ satisfies
$A\le G'$ if and only if $A\prec G$. Thus we may assume any set
of generators is a family. 

For a family $\A$, consider the function $f_{\A}$ defined by 
\[
f_{\A}(x_{1},\ldots,x_{n})=\sum_{B\in\A}\prod_{i\in B}x_{i}.
\]
Notice that $f_{\A}(1,\ldots,1)=|\A|$. Let 
\[
\delta_{i,X}=\begin{cases}
1 & i\notin X\\
0 & i\in X
\end{cases}.
\]
We define $f_{\A}(x_{1},\ldots x_{n})|_{X=0}=f_{\A}(\delta_{1,X}x_{1},\delta_{2,X}x_{2},\ldots,\delta_{n,X}x_{n})$.
Notice that 
\[
f_{\A}(x_{1},\ldots,x_{n})-f_{\A}(x_{1},\ldots x_{n})|_{X=0}=f_{\A(X)}(x_{1},\ldots,x_{n}).
\]
Recall as discussed in Section \ref{s.intcond}, we may assume any
set of generators is a family, so throughout this section, we assume
all generators have $r$ elements.

For $\A$ of the form $\A=F(r,n,\{\{a_{1},\ldots,a_{r}\}\})$ we denote
$f_{\A}$ by $f_{a_{1},\ldots,a_{r}}$. Proposition \ref{p.genfunc}
gives a recursive method for computing $f_{a_{1},\ldots,a_{n}}$.

\begin{prop}\label{p.genfunc} For a compressed family $\A=F(r,n,\{\{a_{1},\ldots,a_{r}\}\})$,
we have
\[
f_{a_{1},\ldots,a_{r}}(x_{1},\ldots,x_{n})=\sum_{i=r}^{a_{r}}x_{i}f_{\{\min\{a_{j},i+j-r\}\}_{j=1}^{r-1}}(x_{1},\ldots,x_{n-1}),
\]
where $\{\min\{a_{j},i+j-r\}\}_{j=1}^{r-1}$ denotes the $r-1$ element
set where the $j$-th element is $\min\{a_{j},i+j-r\}$.\end{prop}

\begin{proof}

Consider $B\in\A=F(r,n,\{\{a_{1},\ldots,a_{r}\}\})$ with largest
element $i$. Notice we must have $r\le i\le a_{r}$, and $b_{j}\le a_{j}$
because $B\le\{a_{1},\ldots,a_{r}\}$. Since $b_{j}\le b_{j+1}-1$,
and the largest element of $B$ is $i$, the $j$-th element is at
most $i+j-r$. Combining these observations, we have $\{b_{1},\ldots,b_{r-1}\}\le\{\min\{a_{j},i+j-r\}\}_{j=1}^{r-1}$.
Conversely, every $B$ such that $b_{r}=i$ and $\{b_{1},\ldots,b_{r-1}\}\le\min\{a_{j},i+j-r\}$
is in $\A$ since $\A$ is compressed, which gives the desired expression.\end{proof}

\begin{rem} Observe that $B\le G=\{g_{1},\ldots,g_{r}\}$ and $B\le H=\{h_{1},\ldots,h_{r}\}$
if and only if $B\le\{\min\{g_{i},h_{i}\}\}_{i=1}^{r}$. Thus by the
preceding proposition, we may obtain $f_{\A}$ with $\A=F(r,n,\mathcal{G})$
for any set of generators $\G$ using the principle of inclusion-exclusion.
For example with $\G=\{G,H\}$,
\[
f_{F(r,n,\{G,H\})}=f_{F(r,n,\{G\})}+f_{F(r,n,\{H\})}-f_{F(r,n,\{\{\min\{g_{i},h_{i}\}\}_{i=1}^{r}\})}.
\]
\end{rem}

\section{Concluding Remarks}

Theorems \ref{t.multgen} and \ref{t.onegen} each serve an important
purpose. Previously, there was no known relation between $n$ and
$r$ that guarantees that one of the eventually EKR sets classified
by Barber is EKR. Theorem \ref{t.multgen} gives such a bound. However,
it is unlikely that the bound given in Theorem \ref{t.multgen} is
tight, and Theorem \ref{t.onegen} gives a suggestion for the optimal
bound.

Though the proofs of Theorems \ref{t.multgen} and \ref{t.onegen}
are long, the underlying ideas, which are outlined at the beginning
of Section \ref{s.onegenprelims} and at the beginning of the proof
of Theorem \ref{t.multgen}, are simple. However, it seems difficult
to use these methods to give a proof of Conjecture \ref{c.phicon},
since giving explicit formulas for $|F(r,n,\G)(X)|$ when $\G$ contains
many generators can be quite complicated. Furthermore, the methods
of this paper rely heavily on computer computations, and we hope that
a different proof may be given. An answer to the following question,
a slight variant of a question originally asked in \cite{barber}
may provide such a proof of Conjecture \ref{c.phicon}. 

\begin{quest}

Given $X$, is there a short list of families, one of which maximizes
$|\A(X)|$?

\end{quest}

A natural choice for such a list is $\mathfrak{A}=\{\A_{n,r,s}\}_{s=1}^{r}$,
and this list would be especially useful because of Theorem \ref{t.onegen}.
However, this does not hold in general, for example when $X=\{4,r+2\}$,
$r=5$, and $n=11$, we have 
\begin{eqnarray*}
|\A_{n,r,1}(X)|=|\S(X)| & = & 140,\\
|\A_{n,r,2}(X)| & = & 121,\\
|\A_{n,r,3}(X)| & = & 136,\\
|\A_{n,r,4}(X)| & = & 140,\\
|\A_{n,r,5}(X)| & = & 105,
\end{eqnarray*}
but 
\[
|F(r,n,\{\{2,3,4\},\{3,4,6,7\}\})(X)|=142.
\]
One may still hope that $\mathfrak{A}$ provides such a list for certain
values of $n$ and $r$. 

Another possible direction for research concerns $t$-intersecting
families. We say a family $\A$ is $t$-intersecting if for any $A,B\in\A$,
we have $|A\cap B|\ge t$. We ask,

\begin{quest}

Can our results be generalized to $t$-intersecting families?

\end{quest}

\noindent For $t$-intersecting families, \cite{barber} suggests
considering
\[
\A(s,X)=\{A\in\A:|A\cap X|\ge s\}
\]
and asking for which $X$ do we have $|\A(s,X)|\le|\S(s,X)|$ for
all compressed and $t$-intersecting $\A$. We suspect it is possible
to use similar techniques to those used in \cite{barber} and this
paper to obtain partial results in this more general case.

\section*{Acknowledgments}

This research was done at the University of Minnesota Duluth math
REU, which is run by Joe Gallian, to whom the author is thankful.
The REU was supported by the National Science Foundation and the Department
of Defense (grant number 1062709) and the National Security Agency
(grant number H98230-11-1-0224). The author would like to thank program
advisors Eric Riedl, Davie Rolnick, and Adam Hesterberg for many helpful
discussions, as well as for their comments on this paper. The author
would also like to thank Duluth visitor Jonathan Wang for comments
on this paper.

\end{document}